\documentclass{article}%
\usepackage{amsfonts}
\usepackage{amssymb}
\usepackage{amsmath}
\usepackage{graphicx}
\usepackage[latin1]{inputenc}

\newtheorem{theorem}{Theorem}[section]
\newtheorem{acknowledgment}[theorem]{Acknowledgment}

\newtheorem{example}[theorem]{Example}

\newtheorem{lemma}[theorem]{Lemma}

\newenvironment{proof}[1][Proof]{\noindent\textbf{#1.} }{\ \rule{0.5em}{0.5em}}
\begin{document}

\title{A MAXIMUM PRINCIPLE FOR INFINITE HORIZON DELAY EQUATIONS}
\author{\ N.AGRAM\thanks{Laboratory of Applied Mathematics, University Med Khider, Po.
Box 145, Biskra $\left(  07000\right)  $ Algeria.}, S.HAADEM\thanks{ Center of
Mathematics for Applications (CMA), University of Oslo, Box 1053 Blindern,
N-0316 Oslo, Norway.}, B. ØKSENDAL\thanks{Center of Mathematics for
Applications (CMA), University of Oslo, Box 1053 Blindern, N-0316 Oslo,
Norway. } \thanks{The research leading to these results has received funding
from the European Research Council under the European Community's Seventh
Framework Programme (FP7/2007-2013) / ERC grant agreement no [228087].} and
F.PROSKE\thanks{Department of Mathematics, University of Oslo, Box 1053
Blindern, N-0316 Oslo, Norway.}. \ \ \ \ \ \ }
\date{\ \ 28 June 2012}
\maketitle

\textsf{Keywords: Infinite horizon; Optimal control; Stochastic delay equation; L\'{e}vy processes; Maximum principle; Hamiltonian;  Adjoint process; Partial information. \newline\newline
2010 Mathematics Subject Classification: \newline
Primary 93EXX; 93E20; 60J75; 34K50 \newline
Secondary 60H10; 60H20; 49J55}

\begin{abstract}
We prove a maximum principle of optimal control of stochastic delay equations
on infinite horizon. We establish first and second sufficient stochastic
maximum principles as well as necessary conditions for that problem. We
illustrate our results by an application to the optimal consumption rate from
an economic quantity.

\end{abstract}

\section{Introduction}

To solve the stochastic control problems, there are two approaches: The dynamic
programming method (HJB equation) and the maximum principle.

In this paper, our system is governed by the stochastic differential delay
equation (SDDE in short):%

\begin{equation*}
\begin{cases}
dX(t) =b\left(  t,X(t),Y(t),A(t),u(t)\right)dt\\
+ \sigma\left(t,X(t),Y(t),A(t),u(t)\right)  dB(t)\\
+ \int_{\mathbb{R}_0} \theta\left(  t,X(t),Y(t),A(t),u(t),z\right)\overset{\sim}{N}(dt,dz);
& t\in\left[  0,\infty\right) \\
X(t) =X_{0}(t); &t\in [-\delta,0] \\
Y(t) =X(t-\delta) &t\in [0,\infty)\\
A(t)=\int_{t-\delta}^{t}e^{-\rho(t-r)}X(r)dr &t\in [0,\infty)
\end{cases}
\end{equation*}
which maximise the functional
\begin{equation}
\begin{array}
[c]{c}%
J(u)=E\left[
{\displaystyle\int\limits_{0}^{\infty}}
f\left(  t,X(t),Y(t)\text{, }A(t)\text{, }u(t)\right)  \text{ }dt\right]
\end{array}
\tag{$1.2$}%
\end{equation}

where $u(t)$ is the control process.

The SDDE is not Markovian so we cannot use the dynamic programming method.However, we will prove stochastic maximum principles for this problem.

A sufficient maximum principle in infinite horizon whitout with non-trivial transvertility conditions where treated by \cite{HOP}. The 'natural' transversality condition in the infinite case would be a zero limit condition, meaning in the economic sense that one more unit of good at the limit gives no additional value. But this property is not necessarily verified. In fact \cite{halkin} provides a counterexample for a 'natural' extension of the finite-horizon transversality conditions. Thus some care is needed in the infinite horizon case. For the case of the 'natural' transversality condition the discounted control problem was studied by \cite{MV}.

In real life, delay occurs everywhere in our society. For example this is the case in biology
where the population growth depends not only on the current population size
but also on the size some time ago. The same situation may occur in many
economic growth models.

The stochastic maximum principle with delay has been studied by many authors. 
For example, \cite{EOS} proved a verification theorem of variational inequality.
\cite{OS} established the sufficient maximum principle for certain class of
stochastic control systems with delay in the state variable. In \cite{HOP} they studied inifinite horizon bu withou a delay. In \cite{CW}, they derived a stochastic maximum principle for a system with delay both in the state variable and the control variable. In \cite{OSZ} they studied the finite horizon version of this paper, however, to our knowledge, no one has studied the infinite horizon case for delay equations.

In our paper, we establish two sufficient maximum principles and one necessary
for the stochastic delay systems on infinite horizon with jumps.

For backward differential equations see \cite{Situ}, \cite{LP}
For the infinite horizon BSDE see \cite{PS}, \cite{Pardoux}, \cite{Yin}, \cite{BBP} and \cite{BSDE}.For more details about jump diffusion markets see \cite{OS2} and for background and details about stochastic fractional delay equations see \cite{Mohammed}. 

Our paper is organised as follows: In the second section, we formulate the
problem. The third section is devoted to the first and second sufficient maximum
principles with an application to the optimal consumption rate from an economic
quantity described by a stochastic delay equation. In the fourth section, we
formulate a necessary maximum principle and we prove an existence and
uniqueness of the advanced backward stochastic differential equations on
infinite horizon with jumps in the last section.

\section{Formulation of the problem}

Let $\left(  \Omega,\mathcal{F},P\right)  $ be a probability space with
filtration $\mathcal{F}_{t}$ satisfying the usual conditions, on which an
$\mathbb{R}$-valued standard Brownian motion $B\left(  .\right)  $ and an
independent compensated Poisson random measure $\overset{\sim}{N}(dt$,
$dz)=N(dt$, $dz)-\nu$ $(dz)$ $dt$ are defined.

We consider the following stochastic control system with delay :%
\begin{equation}
\left\{
\begin{array}
[c]{l}%
dX(t)=b\left(  t,X(t),Y(t)\text{, }A(t)\text{, }u(t)\right)  dt+\sigma\left(
t,X(t),Y(t)\text{, }A(t)\text{, }u(t)\right)  dB(t)\\
\text{ \ \ \ \ \ \ \ \ \ \ \ \ \ \ \ \ }+%
{\displaystyle\int\limits_{\mathbb{R}_0}}
\theta\left(  t,X(t),Y(t)\text{, }A(t)\text{, }u(t)\text{, }z\right)
\overset{\sim}{N}(dt,dz);t\in\left[  0,\infty\right) \\
X(t)=X_{0}(t);\text{ \ \ \ \ \ \ \ \ \ \ \ \ \ \ \ \ \ \ \ }t\in\left[
-\delta,0\right]  \text{\ \ }\\
Y(t)=X(t-\delta)\ \ \ \ \ \ \ \ \ \ \ \ \ \ \ \ A(t)=%
{\displaystyle\int\limits_{t-\delta}^{t}}
e^{-\rho(t-r)}X(r)dr
\end{array}
\right. \tag{$2.1$}%
\end{equation}%
\[%
\begin{array}
[c]{l}%
\delta>0\text{, }\rho>0\text{ are given constants.}\\
b:[0,\infty)\times\mathbb{R}\times\mathbb{R}\times\mathbb{R}\times
\mathcal{U}\times\Omega\rightarrow\mathbb{R},\\
\sigma:[0,\infty)\times\mathbb{R}\times\mathbb{R}\times\mathbb{R}%
\times\mathcal{U}\times\Omega\rightarrow\mathbb{R},\\
\theta:[0,\infty)\times\mathbb{R}\times\mathbb{R}\times\mathbb{R}%
\times\mathcal{U}\times\mathbb{R}_{0}\times\Omega\rightarrow\mathbb{R},
\end{array}
\]

are given functions such that for all $t$, $b(t$, $x$, $y$, $a$, $u$, .$)$,
$\sigma(t$, $x$, $y$, $a$, $u$, .$)$ and $\theta(t$, $x$, $y$, $a$, $u$,
$z$.$)$ are $\mathcal{F}_{t}$-mesurable for all $x\in%
\mathbb{R}
$, $y\in \mathbb{R}$, $a\in%
\mathbb{R}
$, $u\in\mathcal{U}$ and $z\in\mathbb{R}_{0}$.

\bigskip Let $\mathcal{E}_{t}\subset\mathcal{F}_{t}$ be a given subfiltration,
representing the information available to the controller at time $t$.

Let $\mathcal{U}$ be a non-empty subset of $\mathbb{R}.$ We let
$\mathcal{A}_{\mathcal{E }}$ denote the family of admissible $\mathcal{E}_{t}%
$-adapted control processes.

An element of $\mathcal{A}_{\mathcal{E }}$ is called an admissible control.

The corresponding performance functional is%
\begin{equation}
J(u)=E\left[
{\displaystyle\int\limits_{0}^{\infty}}
f\left(  t,X(t),Y(t)\text{, }A(t)\text{, }u(t)\right)  \text{ }dt\right]
;u\in\mathcal{A}_{\mathcal{E}},\tag{$2.2$}%
\end{equation}

where we assume that%

\begin{equation}
E%
{\displaystyle\int\limits_{0}^{\infty}}
\left\{  \left\vert f\left(  t,X(t),Y(t)\text{, }A(t)\text{, }u(t)\right)
\text{ }\right\vert +\left\vert \frac{\partial f}{\partial x_{i}}\left(
t,X(t),Y(t)\text{, }A(t)\text{, }u(t)\right)  \text{ }\right\vert
^{2}\right\}  dt<\infty\tag{$2.3$}%
\end{equation}

The value function $\Phi$ is defined as%
\begin{equation}
\Phi(X_{0})=\underset{u\in\mathcal{A}_{\mathcal{E}}}{\text{sup}}%
J(u)\tag{$2.4$}%
\end{equation}

An admissible control $u^{\ast}\left(  .\right)  $ is called an optimal
control for$\ \left(  2.1\right)  $\ if it attains the maximum of $J\left(
u\left(  .\right)  \right)  $ over $\mathcal{A}_{\mathcal{E}}.$ $\left(
2.1\right)  $ is called the state equation, the solution $X^{\ast}$
corresponding to $u^{\ast}\left(  .\right)  $ is called an optimal trajectory.

\section{A \textbf{sufficient maximum principle}}

Our objective is to establish a sufficient maximum principle.

\subsection{Hamiltonian and time-advanced BSDEs for adjoint equations}

We now introduce the adjoint equations and the Hamiltonian function for our problem.

The Hamiltonian is%

\begin{equation}%
\begin{array}
[c]{c}%
H(t,x,y,a,u,p,q,r(.))=f(t,x,y,a,u)+b(t,x,y,a,u)p+\sigma(t,x,y,a,u)q\\
+%
{\displaystyle\int\limits_{\mathbb{R} _{0}}}
\theta(t,x,y,a,u,z)r(z)\nu(dz)\text{,}%
\end{array}
\tag{$3.1$}%
\end{equation}

where%

\[
H:[0,\infty)\times%
\mathbb{R}
\times%
\mathbb{R}
\times%
\mathbb{R}
\times U\times%
\mathbb{R}
\times%
\mathbb{R}
\times\mathcal{\Re}\times\Omega\rightarrow%
\mathbb{R}%
\]

and $\mathcal{\Re}$ is the set of functions $r$: $%
\mathbb{R}
_{0}\rightarrow%
\mathbb{R}
$ such that the terms in $(3.1)$ converges and $U$ is the set of possible
control values.

We suppose that $b$, $\sigma$ and $\theta$ are $C^{1}$ functions with respect
to $(x,y,a,u)$ and that%

\begin{equation}%
\begin{array}
[c]{c}%
E\left[
{\displaystyle\int\limits_{0}^{\infty}}
\left\{  \left\vert \dfrac{\partial b}{\partial x_{i}}\left(
t,X(t),Y(t)\text{, }A(t)\text{, }u(t)\right)  \text{ }\right\vert
^{2}+\left\vert \dfrac{\partial\sigma}{\partial x_{i}}\left(
t,X(t),Y(t)\text{, }A(t)\text{, }u(t)\right)  \text{ }\right\vert ^{2}\right.
\right. \\
\left.  \left.  +%
{\displaystyle\int\limits_{\mathbb{R} _{0}}}
\left\vert \dfrac{\partial\theta}{\partial x_{i}}\left(  t,X(t),Y(t)\text{,
}A(t)\text{, }u(t)\right)  \text{ }\right\vert ^{2}\nu(dz)\right\}  dt\right]
<\infty
\end{array}
\tag{$3.2$}%
\end{equation}

\bigskip for $x_{i}=x$, $y$, $a$ and $u.$

The adjoint\ processes $(p(t),q(t),r(t,z)),$ $t\in\left[  0,\infty\right)$, $z$ $\in \mathbb{R}$ are assumed to satisfy the equation :%

\begin{equation}%
\begin{array}
[c]{c}%
dp(t)=E\left[  \mu(t)\mid\mathcal{F}_{t}\right]  dt+q(t)dB_{t}+%
{\displaystyle\int\limits_{\mathbb{R} _{0}}}
r(t,z)\overset{\sim}{N}(dt,dz);t\in\left[  0,\infty\right)  \text{,}%
\end{array}
\tag{$3.3$}%
\end{equation}

where%

\begin{align}
&
\begin{array}
[c]{c}%
\mu(t)=-\dfrac{\partial H}{\partial x}\left(  t,X_{t},Y_{t}\text{, }%
A_{t}\text{, }u_{t},p(t),q(t),r(t,.)\right)
\end{array}
\nonumber\\
&
\begin{array}
[c]{c}%
-\dfrac{\partial H}{\partial y}\left(  t+\delta,X_{t+\delta},Y_{t+\delta
}\text{, }A_{t+\delta}\text{, }u_{t+\delta},p(t+\delta),q(t+\delta
),r(t+\delta,.)\right)
\end{array}
\nonumber\\
&
\begin{array}
[c]{c}%
-e^{\rho t}\left(
{\displaystyle\int\limits_{t}^{t+\delta}}
\dfrac{\partial H}{\partial a}\left(  s,X_{s},Y_{s}\text{, }A_{s}\text{,
}u_{s},p(s),q(s),r(s,.)\right)  e^{-\rho s}ds\right)
\end{array}
\tag{$3.4$}%
\end{align}

\subsection{A first \textbf{sufficient maximum principle}}

\begin{theorem}
Let $\hat{u}\in\mathcal{A}_{\mathcal{E}}$ with corresponding state processes
$\hat{X}(t)$, $\hat{Y}(t)$ and $\overset{\wedge
}{A}(t)$\ and adjoint processes $\hat{p}(t)$, $\overset{\wedge
}{q}(t)$ and $\hat{r}(t,z)$ assumed to satisfy the ABSDE $(3.3)$-
$(3.4)$. Suppose that the following assertions hold:

(i)$%
\begin{array}
[c]{c}%
E\left[  \overline{\underset{t\rightarrow\infty}{\lim}}\mathbf{\ }%
\hat{p}(t)(X(t)-\hat{X}(t))\right]  \geq0.
\end{array}
$

(ii) The function
\[%
\begin{array}
[c]{c}%
(x,y,a,u)\rightarrow H(t,x,y,a,u,\hat{p},\overset{\wedge
}{q},\hat{r}(t.)),
\end{array}
\]

is concave for each $t\in\left[  0,\infty\right)  $ a.s.

(iii)
\begin{align}
&
\begin{array}
[c]{c}%
E\left[
{\displaystyle\int\limits_{0}^{\infty}}
\left\{  \hat{p^{2}}(t)\left(  \sigma^{2}(t)+%
{\displaystyle\int\limits_{\mathbb{R} _{0}}}
\theta^{2}(t,z)\nu(dz)\right)  \right.  \right.
\end{array}
\nonumber\\
&
\begin{array}
[c]{c}%
\text{ \ \ \ \ \ \ \ \ \ }\ \left.  \left.  +X^{2}(t)\left(  \overset{\wedge
}{q^{2}}(t)+%
{\displaystyle\int\limits_{\mathbb{R} _{0}}}
\hat{r}(t,z)\nu(dz)\right)  \right\}  dt\right]  <\infty,
\end{array}
\tag{$3.5$}%
\end{align}

for all $u\in\mathcal{A}_{\mathcal{E }}$.

(iiii)%
\begin{align*}
&
\begin{array}
[c]{c}%
\underset{v\in\mathcal{A}_{\mathcal{E }}}{\max}E\left[  H(t,\text{
}\hat{X}(t),\hat{X}(t-\delta),\overset{\wedge
}{A}(t),v,\hat{p}(t),\hat{q}(t),\overset{\wedge
}{r}(t,.)\mid{\mathcal{E }}_{t}\right]
\end{array}
\\
&
\begin{array}
[c]{c}%
=E\left[  H(t,\text{ }\hat{X}(t),\hat{X}%
(t-\delta),\hat{A}(t),\hat{u}(t),\hat{p}%
(t),\hat{q}(t),\hat{r}(t,.)\mid{\mathcal{E }}%
_{t}\right],
\end{array}
\end{align*}

for all $t\in\left[  0,\infty\right)  $ a.s.

Then $\hat{u}(t)$ is an optimal control for the problem $(2.4)$.
\end{theorem}

\begin{proof}
Choose an arbitrary $u\in\mathcal{A}_{\mathcal{E}}$\ , and consider
\begin{equation}%
\begin{array}
[c]{c}%
J(u)-J(\hat{u})=I_{1}%
\end{array}
\tag{$3.6$}%
\end{equation}

where
\begin{equation}%
\begin{array}
[c]{c}%
I_{1}=E\left[
{\displaystyle\int\limits_{0}^{\infty}}
\left\{  f\left(  t,X(t),Y(t)\text{, }A(t)\text{, }u(t)\right)  -f\left(
t,\hat{X}(t),\hat{Y}(t)\text{, }\overset{\wedge
}{A}(t)\text{, }\hat{u}(t)\right)  \right\}  dt\text{ }\right].
\end{array}
\tag{$3.7$}%
\end{equation}

By the definition $(3.1)$ of $H$\ and the concavity, we have
\begin{align}
&
\begin{array}
[c]{c}%
I_{1}\leq E\left[
{\displaystyle\int\limits_{0}^{\infty}}
\left\{  \dfrac{\partial\hat{H}}{\partial x}%
(t)(X(t)-\hat{X}(t))+\dfrac{\partial\hat{H}}{\partial
y}(t)(Y(t)-\hat{Y}(t))+\dfrac{\partial\hat{H}%
}{\partial a}(t)(A(t)-\hat{A}(t))\right.  \right.
\end{array}
\nonumber\\
&
\begin{array}
[c]{c}%
\text{ \ \ \ }+\dfrac{\partial\hat{H}}{\partial u}(t)(u(t)-\hat
{u}(t))-(b(t)-\hat{b}(t))\hat{p}(t)-(\sigma
(t)-\hat{\sigma}(t))\hat{q}(t)\ \ \ \\
\ \left.  \left.  -%
{\displaystyle\int\limits_{\mathbb{R}_0}}
(\theta(t,z)-\hat{\theta}(t,z))\hat{r}(t,z)\nu
(dz)\right\}  dt\right]  \text{,}%
\end{array}
\tag{$3.8$}%
\end{align}

where we have used the simplified notation%
\[%
\begin{array}
[c]{c}%
\dfrac{\partial\hat{H}}{\partial x}(t)=\dfrac{\partial
\hat{H}}{\partial x}\left(  t,\hat{X}_{t}%
,\hat{Y}_{t}\text{, }\hat{A}_{t}\text{, }%
\hat{u}_{t},\hat{p}(t),\hat{q}%
(t),\hat{r}(t,.)\right).%
\end{array}
\]

Applying the It\^{o} formula to $%
\begin{array}
[c]{c}%
\mathbf{\ }\hat{p}(t)(X(t)-\hat{X}(t))
\end{array}
$ we get%
\begin{align}
&
\begin{array}
[c]{c}%
0\leq E\left[  \overline{\underset{T\rightarrow\infty}{\lim}}\mathbf{\ \ }%
\hat{p}(T)(X(T)-\hat{X}(T))\right] \\
=E\left[  \overline{\underset{T\rightarrow\infty}{\lim}}\mathbf{\ }\left(
{\displaystyle\int\limits_{0}^{T}}
(b(t)-\hat{b}(t))\hat{p}(t)dt+%
{\displaystyle\int\limits_{0}^{T}}
(X(t)-\hat{X}(t))E\left[  \hat{\mu}(t)\mid
\mathcal{F}_{t}\right]  dt\right.  \right.
\end{array}
\nonumber\\
&
\begin{array}
[c]{c}%
\text{\ \ }\left.  \text{\ \ }\left.  +%
{\displaystyle\int\limits_{0}^{T}}
(\sigma(t)-\hat{\sigma}(t))\hat{q}(t)dt+%
{\displaystyle\int\limits_{0}^{T}}
{\displaystyle\int\limits_{\mathbb{R}_0}}
(\theta(t,z)-\hat{\theta}(t,z))\hat{r}(t,z)\nu
(dz)dt\right)  \right]
\end{array}
\nonumber\\
&
\begin{array}
[c]{c}%
=E\left[  \overline{\underset{T\rightarrow\infty}{\lim}}\mathbf{\ }\left(
{\displaystyle\int\limits_{0}^{T}}
(b(t)-\hat{b}(t))\hat{p}(t)dt+%
{\displaystyle\int\limits_{0}^{T}}
(X(t)-\hat{X}(t))\hat{\mu}(t)dt\right.  \right.
\end{array}
\nonumber\\
& \text{ \ \ \ \ \ \ \ \ \ \ \ \ \ \ \ \ \ \ \ \ \ \ }%
\begin{array}
[c]{c}%
+%
{\displaystyle\int\limits_{0}^{T}}
(\sigma(t)-\hat{\sigma}(t))\hat{q}(t)dt
\end{array}
\nonumber\\
&
\begin{array}
[c]{c}%
\text{\ }\left.  \text{\ }\left.  \text{\ }+%
{\displaystyle\int\limits_{0}^{T}}
{\displaystyle\int\limits_{\mathbb{R}_0}}
(\theta(t,z)-\hat{\theta}(t,z))\hat{r}(t,z)\nu
(dz)dt\right)  \right].
\end{array}
\tag{$3.9$}%
\end{align}

Using the definition $(3.4)$ of $\mu$ we see that%
\begin{equation}%
\begin{array}
[c]{c}%
E\left[  \overline{\underset{T\rightarrow\infty}{\lim}}\mathbf{\ }\left(
{\displaystyle\int\limits_{0}^{T}}
(X(t)-\hat{X}(t))\hat{\mu}(t)dt\right)  \right]
=E\left[  \overline{\underset{T\rightarrow\infty}{\lim}}\mathbf{\ }\left(
{\displaystyle\int\limits_{\delta}^{T+\delta}}
(X(t-\delta)-\hat{X}(t-\delta))\hat{\mu}%
(t-\delta)dt\right)  \right] \\
=E\left[  \overline{\underset{T\rightarrow\infty}{\lim}}\mathbf{\ }\left(  -%
{\displaystyle\int\limits_{\delta}^{T+\delta}}
\dfrac{\partial\hat{H}}{\partial x}(t-\delta)(X(t-\delta
)-\hat{X}(t-\delta))dt\right.  \right. \\
-%
{\displaystyle\int\limits_{\delta}^{T+\delta}}
\dfrac{\partial\hat{H}}{\partial y}\left(  t\right)
(Y(t)-\hat{Y}(t))dt-%
{\displaystyle\int\limits_{\delta}^{T+\delta}}
\left(
{\displaystyle\int\limits_{t-\delta}^{t}}
\dfrac{\partial\hat{H}}{\partial a}\left(  s\right)  e^{-\rho
s}ds\right) \\
\left.  \left.  e^{\rho(t-\delta)}(X(t-\delta)-\hat{X}%
(t-\delta))\right)  dt\right]
\end{array}
\tag{$3.10$}%
\end{equation}

Using integration by parts and substituting $r=t-\delta,$ we obtain%
\begin{align}
&
\begin{array}
[c]{c}%
{\displaystyle\int\limits_{0}^{T}}
\dfrac{\partial\hat{H}}{\partial a}(s)(A(s)-\overset{\wedge
}{A}(s))ds
\end{array}
\nonumber\\
&
\begin{array}
[c]{c}%
=%
{\displaystyle\int\limits_{0}^{T}}
\dfrac{\partial\hat{H}}{\partial a}(s)%
{\displaystyle\int\limits_{s-\delta}^{s}}
e^{-\rho(s-r)}(X(r)-\hat{X}(r))dr\text{ }ds
\end{array}
\nonumber\\
&
\begin{array}
[c]{c}%
=%
{\displaystyle\int\limits_{0}^{T}}
\left(
{\displaystyle\int\limits_{r}^{r+\delta}}
\dfrac{\partial\hat{H}}{\partial a}(s)e^{-\rho s}ds\right)
e^{\rho r}(X(r)-\hat{X}(r))\text{ }dr
\end{array}
\nonumber\\
&
\begin{array}
[c]{c}%
=%
{\displaystyle\int\limits_{\delta}^{T+\delta}}
\left(
{\displaystyle\int\limits_{t-\delta}^{t}}
\dfrac{\partial\hat{H}}{\partial a}\left(  s\right)  e^{-\rho
s}ds\right)  e^{\rho(t-\delta)}(X(t-\delta)-\hat{X}(t-\delta))dt
\end{array}
\tag{$3.11$}%
\end{align}

Combining $(3.9)$, $(3.10)$ and $(3.11)$; we get
\begin{align}
&
\begin{array}
[c]{c}%
0\leq E\left[  \overline{\underset{T\rightarrow\infty}{\lim}}\mathbf{\ \ }%
\hat{p}(T)(X(T)-\hat{X}(T))\right]  =E\left[  \left(
{\displaystyle\int\limits_{0}^{\infty}}
(b(t)-\hat{b}(t))\hat{p}(t)dt\right.  \right. \\
-%
{\displaystyle\int\limits_{0}^{\infty}}
\dfrac{\partial\hat{H}}{\partial x}\left(  t\right)
(X(t)-\hat{X}(t))dt-%
{\displaystyle\int\limits_{0}^{\infty}}
\dfrac{\partial\hat{H}}{\partial y}\left(  t\right)
(Y(t)-\hat{Y}(t))dt\\
-%
{\displaystyle\int\limits_{0}^{\infty}}
\dfrac{\partial\hat{H}}{\partial a}\left(  t\right)
(A(t)-\hat{A}(t))dt+%
{\displaystyle\int\limits_{0}^{\infty}}
(\sigma(t)-\hat{\sigma}(t))\hat{q}(t)dt
\end{array}
\nonumber\\
&
\begin{array}
[c]{c}%
\left.  \text{\ }\left.  \text{\ }+%
{\displaystyle\int\limits_{0}^{\infty}}
{\displaystyle\int\limits_{\mathbb{R}_0}}
(\theta(t,z)-\hat{\theta}(t,z))\hat{r}(t,z)\nu
(dz)dt\right)  \right]
\end{array}
\tag{$3.12$}%
\end{align}

Subtracting and adding $%
\begin{array}
[c]{c}%
{\displaystyle\int\limits_{0}^{\infty}}
\dfrac{\partial\hat{H}}{\partial u}(t)(u(t)-\hat{u}(t))dt
\end{array}
$ in $(3.12)$ we conclude%

\[%
\begin{array}
[c]{c}%
0\leq E\left[  \overline{\underset{T\rightarrow\infty}{\lim}}\mathbf{\ \ }%
\hat{p}(T)(X(T)-\hat{X}(T))\right]  =E\mathbf{\ }%
\left[  \left(
{\displaystyle\int\limits_{0}^{\infty}}
(b(t)-\hat{b}(t))\hat{p}(t)dt\right.  \right. \\
-%
{\displaystyle\int\limits_{0}^{\infty}}
\dfrac{\partial\hat{H}}{\partial x}\left(
t,X(t),Y(t),A(t),u(t),p(t),q(t),r(t,.)\right)  (X(t)-\overset{\wedge
}{X}(t))dt\\
-%
{\displaystyle\int\limits_{0}^{\infty}}
\dfrac{\partial\hat{H}}{\partial y}\left(  t,X(t),Y(t)\text{,
}A(t)\text{, }u(t),p(t),q(t),r(t,.)\right)  (Y(t)-\hat{Y}(t))dt\\
-%
{\displaystyle\int\limits_{0}^{\infty}}
\dfrac{\partial\hat{H}}{\partial a}\left(  t,X(t),Y(t)\text{,
}A(t)\text{, }u(t),p(t),q(t),r(t,.)\right)  (A(t)-\hat{A}(t))dt\\
+%
{\displaystyle\int\limits_{0}^{\infty}}
(\sigma(t)-\hat{\sigma}(t))\hat{q}(t)dt+%
{\displaystyle\int\limits_{0}^{\infty}}
{\displaystyle\int\limits_{\mathbb{R}_0}}
(\theta(t,z)-\hat{\theta}(t,z))\hat{r}(t,z)\nu(dz)dt\\
\left.  \text{\ }\left.  \text{\ }-%
{\displaystyle\int\limits_{0}^{\infty}}
\dfrac{\partial\hat{H}}{\partial u}(t)(u(t)-\hat{u}(t))dt+%
{\displaystyle\int\limits_{0}^{\infty}}
\dfrac{\partial\hat{H}}{\partial u}(t)(u(t)-\hat{u}(t))dt\right)
\right] \\
\leq-I_{1}+E\left[
{\displaystyle\int\limits_{0}^{\infty}}
E\left[  \dfrac{\partial\hat{H}}{\partial u}(t)(u(t)-\hat
{u}(t))\mid{\mathcal{E}}_{t}\right]  dt\right]  \text{.}%
\end{array}
\]

Hence%
\[%
\begin{array}
[c]{c}%
I_{1}\leq E\left[
{\displaystyle\int\limits_{0}^{\infty}}
E\left[  \dfrac{\partial\hat{H}}{\partial u}(t)\mid{\mathcal{E}%
}_{t}\right]  (u(t)-\hat{u}(t))dt\right]  \leq0\text{.}%
\end{array}
\]

Since $u\in\mathcal{A}_{\mathcal{E}}$ was arbitrary, this proves $Theorem$ $1$.
\end{proof}

\subsection{A second sufficient maximum principle}

We extend the result in \cite{OS} to infinite horizon with jump diffusions.

Consider again the system
\begin{equation*}
\begin{cases}
dX(t) =b\left(  t,X(t),Y(t),A(t),u(t)\right)dt\\
+ \sigma\left(t,X(t),Y(t),A(t),u(t)\right)  dB(t)\\
+ \int_{\mathbb{R}_0} \theta\left(  t,X(t),Y(t),A(t),u(t),z\right)\overset{\sim}{N}(dt,dz);
& t\in\left[  0,\infty\right) \\
X(t) =X_{0}(t); &t\in [-\delta,0] \\
Y(t) =X(t-\delta) &t\in [0,\infty)\\
A(t)=\int_{t-\delta}^{t}e^{-\rho(t-r)}X(r)dr &t\in [0,\infty)
\end{cases}
\end{equation*}

Let $X_{t}\in C[-\delta,0]$ be the segment of the path of $X$ from $t-\delta$
to $t$, i.e
\begin{align*}%
X_{t}(s)=X(t+s),
\end{align*}

for $s\in [-\delta,0]$. We now give an It$\bar{o}$ formula which is proved in \cite{EOS} without jumps. Adding
the jump parts are just an easy observation.

\begin{lemma}
The It$\bar{o}$ formula for delay

Consider a function
\begin{equation}%
\begin{array}
[c]{c}%
G(t)=F(t,X(t),A(t))\text{,}%
\end{array}
\tag{$3.13$}%
\end{equation}

where $F$ is a function in $C^{1,2,1}(\mathbb{R}^3)$ and
\[
Y(t) = \int_{-\delta}^0 e^{\lambda s} X(t+s)ds.
\]
then

Then
\begin{align*}
dG(t)=LFdt+\sigma(t,x,y,a,u)\frac{\partial F}{\partial x}dB(t)\\
+\int_{\mathbb{R}_0}
\Bigg\{  F(t,X(t^{-}),A(t^{-}))+\theta (  t,X(t),Y(t),A(t),u,z)   \\
-F(t,X(t^{-}),A(t^{-}))\\
 -\frac{\partial F}{\partial x}(t,X(t^{-}),A(t^{-}))+\theta(t,X(t),Y(t),A(t),u,z)  \Bigg\}  \nu(dz)dt\\
+\int_{\mathbb{R}_0}\{  F(t,X(t^{-}),A(t^{-}))+\theta(  t,X(t),Y(t),A(t),u,z)   \\
  -F(t,X(t^{-}),A(t^{-}))\Bigg\}  \tilde{N}(dt,dz)\\
+[  x-\lambda y-e^{-\lambda\delta}a]  \frac{\partial F}{\partial a}dt
\end{align*}

where%
\begin{align*}
LF=LF(t,x,y,a,u)=\dfrac{\partial F}{\partial t}+b\dfrac{\partial F}{\partial
x}+\dfrac{1}{2}\sigma^{2}\dfrac{\partial^{2}F}{\partial x^{2}}%
\end{align*}

\end{lemma}

Now, define the Hamiltonian, $H: \mathbb{R}_{+} \times \mathbb{R} \times \mathbb{R} \times \mathbb{R} \times U \times \mathbb{R}^3 \times \mathbb{R}^2 \times \mathcal{R} \to \times \mathbb{R}$ as
\begin{align*}
&H(t,x,y,a,u,p,q,r(\cdot))\tag{$3.14$}\\
&= f(t,x,y,a,u) + b(t,x,y,a,u) p_1 + (x-\lambda y - e^{-\lambda \delta}a) p_2 \notag\\
&+ \sigma(t,x,y,a,u)q_1 + \int_{\mathbb{R}_0}\theta(t,x,y,a,u,z)r(z)\nu(dz) \notag
\end{align*}
 where $p =(p_1,p_2,p_3)^T \in \mathbb{R}^3$ and $q=(q_1,q_2) \in \mathbb{R}^2$ For each $u \in \mathcal{A}$ the associated adjoint equations are the following backward stochastic differential equations in the unknown $\mathcal{F}_t$-adapted preocesses $(p(t),q(t),r(t,\cdot))$ given by;
\begin{align*}
dp_1(t) &= - \frac{\partial H}{\partial x}(t,X(t),Y(t),A(t),u(t),p(t),q(t))dt + q_1(t)dB(t)\notag\\
&+ \int_{\mathbb{R}_0}r(t,z) \bar{N}(dt,dz),\\
dp_2(t) &= - \frac{\partial H}{\partial y}(t,X(t),Y(t),A(t),u(t),p(t),q(t))dt + q_2(t)dB(t), \tag{3.15}\\ 
dp_3(t) &= - \frac{\partial H}{\partial a}(t,X(t),Y(t),A(t),u(t),p(t),q(t))dt, \tag{3.16}
\end{align*}

\begin{theorem}[An infinite horizon maximum principle for delay equations]
Suppose $\hat{u} \in \mathcal{A}$ and let $(\hat{X},\hat{Y},\hat{A})$ and $(p(t),q(t),r(t,\cdot))$ be the corresponding solutions of  $(3.15)$-$(3.16)$, repectively. Suppose that
\[
 H(t,\cdot,\cdot,\cdot,\cdot,,p(t),q(t),r(t,\cdot)))
\]
are concave for all $t\geq 0$,
\begin{align*}
E\left[ H(t,\hat{X}(t),\hat{Y}(t),\hat{A}(t),\hat{u}(t),\hat{p}(t),\hat{q}(t),\hat{r}(t,\cdot)) | \mathcal{E}_t\right] \notag \\
= \max_{u \in U} E\left[ H(t,\hat{X}(t),\hat{Y}(t),\hat{A}(t),u,\hat{p}(t),\hat{q}(t),\hat{r}(t,\cdot)) | \mathcal{E}_t\right]. \tag{3.17} 
\end{align*}
Further, assume that
\begin{align*}
 E[\lim \hat{p_1}(t)(X(t) - \hat{X}(t))] \geq 0, \tag{3.18}
\end{align*}
and
\begin{align*}
 E[\lim \hat{p_2}(t)(Y(t) - \hat{Y}(t))] \geq 0. \tag{3.19}
\end{align*}
In addition assume that
\begin{align*}
p_3(t) = 0, \tag{3.20}
\end{align*}
for all $t$.
Then $\hat{u}$ is an optimal control. 
\end{theorem}

\begin{proof}
To simplify notation we put
\[
 \zeta(t) = (X(t),Y(t),A(t)),
\]
and
\[
 \hat{\zeta}(t) = (\hat{X}(t),\hat{Y}(t),\hat{A}(t)).
\]
Let 
\begin{align*}
 I:= E[ \int_0^{\infty} (f(t,\hat{\zeta}(t),\hat{u}(t)) - f(t,\zeta(t),u(t)))dt] 
\end{align*}

Then we have that
\begin{align*}
I &= E[  \int_0^{\infty} (H(t,\hat{\zeta}(t),\hat{u}(t),p(t),q(t),r(t,\cdot)) - H(t,\zeta(t),u(t)),p(t),q(t),r(t,\cdot))dt]\notag\\
&- E[\int_0^{\infty} (b(t,\hat{\zeta}(t),\hat{u}(t)) - b(t,\zeta(t),u(t)))p_1(t)dt]\notag\\
&- E[\int_0^{\infty} \{ (\hat{X}(t) - \lambda\hat{Y}(t) - e^{-\lambda \delta} \hat{A}(t)) - (X(t) - \lambda Y(t) - e^{-\lambda \delta}A(t)) \} p_2(t)dt]\notag\\
&- E[\int_0^{\infty} \{  \sigma(t,\hat{\zeta}(t),\hat{u}(t)) - \sigma(t,\zeta(t),u(t)) \}q_1(t) dt]\notag\\
&- E[\int_0^{\infty} \int_{\mathbb{R}_0} (\theta(t,\hat{\zeta}(t),\hat{u}(t),z) - \theta(t,\zeta,u,z)) \times r(t,z) \nu(dz) dt]\notag\\
&=: I_{1} + I_{2} + I_{3} +I_{4} +I_{5}. \tag{3.21}
\end{align*}

Since $(\zeta,u) \to H(\zeta,u)$ is concave and $(3.12)$, we have that
\begin{align*}
H(\zeta,u) - H(\hat{\zeta},\hat{u}) &\leq H_{\zeta}(\hat{\zeta},\hat{u})\cdot (\zeta - \hat{\zeta}) + H_{u}(\hat{\zeta},\hat{u}) \cdot(u - \hat{u})\\
&\leq H_{\zeta}(\hat{\zeta},\hat{u})\cdot (\zeta - \hat{\zeta}) 
\end{align*}
where $H_{\zeta} = (\frac{\partial H}{\partial x},\frac{\partial H}{\partial y},\frac{\partial H}{\partial a})$. 
From this we get that
\begin{align*}
I_{1}  &\geq E\Bigg[  \int_0^{\infty} -H_{\zeta}(t,\hat{\zeta}(t),\hat{u}(t),p(t),q(t)) \cdot (\zeta(t) - \hat{\zeta}(t))dt\Bigg]\notag\\
& = E\Bigg[  \int_0^{\infty} (\zeta(t) - \hat{\zeta}(t)) dp(t) - \int_0^{\infty} (X(t) - \hat{X}(t))q_1(t)dB(t)\notag\\
&- \int_0^{\infty} (Y(t) - \hat{Y}(t))q_2(t)dB(t)\Bigg] \notag\\
& = E\Bigg[  \int_0^{\infty} (X(t) - \hat{X}(t)) dp_1(t) + \int_0^{\infty}(Y(t) - \hat{Y}(t)) dp_2(t)\Bigg]. \tag{3.22} 
\end{align*}
From $(3.18)$, $(3.19)$ and $(3.20)$ we get that 
\begin{align*}
0 &\geq - E [  \lim \hat{p_1}(t)(X(t) - \hat{X}(t)) + \lim \hat{p_2}(t)(Y(t) - \hat{Y}(t))]\\
&= - E\Bigg[ \int_0^{\infty} (X(t) - \hat{X}(t))dp_1(t) +  \int_0^{\infty}p_1(t)d(X(t) - \hat{X}(t))\\
&+   \int_0^{\infty} \left[\sigma(t,\zeta(t),u(t)) - \sigma(t,\hat{\zeta}(t),\hat{u}(t)\right]q_1(t)dt\\
&+ \int_0^{\infty} \int_{\mathbb{R}_0} (\theta(t,\hat{\zeta}(t),\hat{u}(t),z) - \theta(t,\zeta,u,z)) \times r(t,z) \nu(dz) dt\\
&+ \int_0^{\infty}(Y(t) - \hat{Y}(t))dp_2(t) + \int_0^{\infty}p_2(t)d(Y(t) - \hat{Y}(t))\Bigg].
\end{align*}
Combining this with $(3.21)$ and $(3.22)$ we have that
so that
\begin{align*}
-I = I_{1} + I_{2} + I_{3} +I_{4} +I_{5} \leq 0.
\end{align*}
Hence $J(\hat{u}) - J(u) = I \geq 0$,
and $\hat{u}$ is an optimal control for our problem.

\end{proof}

\begin{example}[A non-delay infinite horizon example]
Let us first consider a non-delay example. Assume we are given
\[
 J(u) = E\left[\int_0^{\infty}e^{-\rho t} \frac{1}{\gamma} (u(t)X(t))^{\gamma}dt\right],
\]
where
\begin{align*}
\begin{cases}\textbf{}
dX(t) &= \left[X(t)\mu  - u(t)X(t)\right]dt\\
&+ \sigma(t,X(t),u(t))dB(t); t\geq 0,\\
X(t) &= X_0
\end{cases}
\end{align*}
$\gamma \in (0,1)$ and $\rho,\delta> 0$.
In this case the Hamiltonian $(3.14)$ takes the form
\begin{align*}
 H(t,u,x,p,q) &= e^{-\rho t} \frac{1}{\gamma} (ux)^{\gamma} + [ x\mu - ux]p_1\\
&+ [x  - e^{-\rho \delta}a]p_2 + \sigma(t,x,y,a,u) q \notag,
\end{align*}
so that we get the partial derivative
\[
 \nabla_u H(t,u,x,p,q) = e^{-\rho t}u^{\gamma -1}x^{\gamma} - xp_1 - \frac{\partial \sigma}{\partial u} q_1.
\]
This gives us that
\begin{align*}
 p_1(t) = e^{-\rho t}x^{\gamma -1}u(t)^{\gamma -1}  - \frac{\partial \sigma}{\partial u} \frac{1}{x} q_1.
\end{align*}
We now see that the adjoint equations are given by:
\begin{align*}
dp_1(t) &=   - \Big[e^{-\rho t}(u(t))^{\gamma}X(t)^{\gamma-1}\\
&+ (\mu - u(t)) p_1(t) + p_2(t) + \frac{\partial \sigma}{\partial x}q_1(t)  \Big]dt + q_1(t)dB(t),\\
dp_2(t) &= -  q_2(t)dB(t),\\
dp_3(t) &= - \left[  - e^{-\rho \delta} p_2(t) + \frac{\partial \sigma}{\partial a}q_1(t)\right]dt.
\end{align*}
Since $p_3(t)$ must be $0$, we then get  $q_1=q_2 =0$. and
\begin{align*}
p_2(t) = 0,
\end{align*}
which gives us that
\begin{align*}
dp_1(t) &=   - \left[e^{-\rho t}(u(t))^{\gamma}X(t)^{\gamma-1}dt +  (\mu - u(t)) p_1(t))\right]dt,\\
dp_2(t) &= 0,
\end{align*}
and
\begin{align*}
p_1(t) = e^{-\rho t}X(t)^{\gamma -1}u(t)^{\gamma -1}.
\end{align*}
So
\begin{align*}
dp_1(t) &=   - \left[e^{-\rho t}(u(t))^{\gamma}X(t)^{\gamma-1}dt + (\mu - u(t)) p_1(t) \right]dt ,\\
&=   -  \mu p_1(t) dt  
\end{align*}
which gives
\begin{align*}
 p_1(t) = p_1(0)e^{- \mu t},
\end{align*}
for some constant $p_1(0)$,so that
\begin{align*}
 u(t) = \frac{p_1(0)^{\frac{1}{\gamma -1}}}{X(t)}e^{\frac{1}{\gamma -1}(\rho t - \mu t )}),
\end{align*}
for all $t>0$.
Inserting $u$ into the dynamics of $X$, we get that
\begin{align*}
 dX(t) = \left[\mu X(t) - p_1(0)^{\frac{1}{\gamma -1}} e^{\frac{1}{\gamma -1}(\rho t - \mu t )}\right]dt.
\end{align*}
So
\begin{align*}
 X(t) = e^{\mu t} \left[X(0) - p_1(0)^{\frac{1}{\gamma -1}}\int_0^t exp((-\mu - \frac{1}{1-\gamma}(\lambda - \mu))s)ds\right].
\end{align*}
To ensure that $X(t)$ is alwasys non-negative, we get the optimal $p(0)$ as
\begin{align*}
 p_1(0) = \left[\frac{X(0)}{\int_0^{\infty} exp((-\mu - \frac{1}{1-\gamma}(\lambda - \mu))s)ds}\right]^{\gamma -1}.
\end{align*}
We now see that $\lim p_1(t) = 0$, so that we have
\[
 E[\lim \hat{p_1}(t)(X(t) - \hat{X}(t))] \geq 0.
\]
This tells us that $\hat{u}$ is an optimal control.
\end{example}

\begin{example}[An infinite horizon example with delay]
Now let us consider a case where we have delay. This is an infinite horizon version of Example 1 in \cite{OS}.
Let
\[
 J(u) = E\left[\int_0^{\infty}e^{-\rho t} \frac{1}{\gamma} (u(t)(X(t) + Y(t)e^{\rho \delta}\beta))^{\gamma}dt\right],
\]
where
\begin{align*}
\begin{cases}\textbf{}
dX(t) &= [X(t)\mu + Y(t)\alpha + \beta A(t) - u(t)(X(t) + Y(t)e^{\rho \delta}\beta)]dt\\
&+ \sigma(t,X(t),Y(t),A(t),u(t))dB(t); t\geq 0,\\
X(t) &= X_0(t); t \in [-\delta,0],
\end{cases}
\end{align*}
$\gamma \in (0,1)$ and $\rho,\delta> 0$.
In this case the Hamiltonian $(3.20)$ takes the form
\begin{align*} 
 H(t,u,x,y,a,p,q) &= e^{-\rho t} \frac{1}{\gamma} (u(x+ye^{\rho \delta}\beta))^{\gamma} + [ x\mu   +\alpha y +\beta a - u(x+ye^{\rho \delta}\beta)]p_1\\
&+ [x - \lambda y - e^{-\rho \delta}a]p_2  + \sigma(t,x,y,a,u) q \notag,
\end{align*}
so that we get the partial derivative
\[
 \nabla_u H(t,u,x,y,a,p,q) = e^{-\rho t}u^{\gamma -1}(x+ye^{\rho \delta}\beta)^{\gamma} - (x+ye^{\rho \delta}\beta)p_1 - \frac{\partial \sigma}{\partial u} q_1.
\]
This gives us that
\begin{align*} 
 p_1(t) = e^{-\rho t}(x+ye^{\rho \delta}\beta)^{\gamma -1}u(t)^{\gamma -1}  - \frac{\partial \sigma}{\partial u} \frac{1}{(x+ye^{\rho \delta}\beta)} q_1.
\end{align*}
We now see that the adjoint equations are given by:
\begin{align*}
dp_1(t) &=   - [e^{-\rho t}(u(t))^{\gamma}(X(t) + Y(t)e^{\rho \delta}\beta)^{\gamma-1}dt\\
&+ (\mu - u(t)) p_1(t) + p_2(t) + \frac{\partial \sigma}{\partial x}q_1(t)  ]dt + q_1(t)dB(t),\\
dp_2(t) &= - [e^{-\rho t}(u(t))^{\gamma}(X(t) + Y(t)e^{\rho \delta}\beta)^{\gamma-1}e^{\rho \delta}\beta dt\\
&+ (\alpha -u(t)e^{\rho \delta}\beta)p_1(t)  - \lambda p_2(t) + \frac{\partial \sigma}{\partial y}q_1(t)  ]dt + q_2(t)dB(t),\\
dp_3(t) &= - [ \beta p_1(t) - e^{-\rho \delta} p_2(t) + \frac{\partial \sigma}{\partial a}q_1(t)]dt.
\end{align*}
Let us try to choose $q_1=q_2 =0$. Since $p_3(t) = 0$, we then get
\begin{align*} 
 p_1(t) = \frac{e^{-\rho \delta }}{\beta} p_2(t),
\end{align*}
which gives us that
\begin{align*}
dp_1(t) &=   - [e^{-\rho t}(u(t))^{\gamma}(X(t) + Y(t)e^{\rho \delta}\beta)^{\gamma-1}dt +  (\mu - u(t)) p_1(t) + e^{\rho \delta}\beta p_1(t)]dt,\\
dp_2(t) &= - [e^{-\rho t}(u(t))^{\gamma}(X(t) + Y(t)e^{\rho \delta}\beta)^{\gamma-1}e^{\rho \delta}\beta dt + (\alpha \frac{e^{-\rho \delta}}{\beta} -u(t))p_2(t)  - \lambda p_2(t)] dt,
\end{align*}
and
\begin{align*} 
p_1(t) = e^{-\rho t}(X(t) + Y(t)e^{\rho \delta}\beta)^{\gamma -1}u(t)^{\gamma -1}
\end{align*}
or
\begin{align*} 
 u(t) = \frac{e^{\frac{\rho t}{\gamma -1}} p_1^{\frac{\rho t}{\gamma -1}}(t)}{X(t) + Y(t)e^{\rho \delta}\beta}. \tag{3.23}
\end{align*}

Hence, to ensure that 
\[
 p_1(t) = \frac{e^{-\rho \delta }}{\beta} p_2(t)
\]
we need that
\begin{align*} 
 \alpha = e^{\rho \delta}\beta (\mu + \lambda + e^{\rho \delta}\beta).
\end{align*}
So
\begin{align*}
dp_1(t) &=   - [e^{-\rho t}(u(t))^{\gamma}(X(t) + Y(t)e^{\rho \delta}\beta)^{\gamma-1}dt + (\mu - u(t)) p_1(t) + e^{\rho \delta}\beta p_1(t) ]dt ,\\
&=   - [ \mu p_1(t) + e^{\rho \delta}\beta p_1(t) ]dt, 
\end{align*}
which gives us that 
\begin{align*}
 p_1(t) = p_1(0)e^{- (\mu + e^{\rho \delta}\beta)t},
\end{align*}
for some constant $p_1(0)$. Hence by $(3.23)$ we get 
\begin{align*}
 u(t) = u_{p_1(0)}= \frac{p_1(0)^{\frac{1}{\gamma -1}}}{(X(t) + Y(t)e^{\rho \delta}\beta)}e^{\frac{1}{\gamma -1}(\rho t - (\mu t + e^{\rho \delta} \beta t))}),
\end{align*}
for all $t>0$ and some $p_1(0)$. In analogy with Example 3.4 it is natural to conjecture that the optimal value, $K$, of $p_1(0)$is given by
\begin{align*}
 K = \inf \{p_1(0): X^{p_1(0)}(t) + Y^{p_1(0)}(t)e^{\lambda \delta}\beta > 0, \text{ for all } t > 0 \}, 
\end{align*}
see \cite{MS}.
So,the optimal control is given by
\begin{align*}
 u(t) = \frac{K^{\frac{1}{\gamma -1}}}{(X(t) + Y(t)e^{\rho \delta}\beta)}e^{\frac{1}{\gamma -1}(\rho t - (\mu t + e^{\rho \delta} \beta t))}).
\end{align*}
From this we get that $\lim p_1(t) = \lim p_2(t) = 0$, so that we have
\[
 E[\lim \hat{p_1}(t)(X(t) - \hat{X}(t))] \geq 0,
\]
and
\[
 E[\lim \hat{p_2}(t)(Y(t) - \hat{Y}(t))] \geq 0.
\]
This tells us that $\hat{u}$ is an optimal control.
\end{example}

\section{\textbf{A necessary maximum principle}$\ \ \ \ \ \ $}

In addition to the assumptions in the previous section, we now assume the following.

$(A_{1})$\ For all $u\in\mathcal{A}_{\mathcal{E}}$ and all $\beta
\in\mathcal{A}_{\mathcal{E}}$ bounded, there exists $\mathcal{\epsilon
}\mathcal{>}0$ such that%

\[%
\begin{array}
[c]{c}%
u+s\beta\in\mathcal{A}_{\mathcal{E}}\text{ \ \ \ \ for all }s\in
(-{ \epsilon}\text{, }\mathcal{\epsilon})\text{.}%
\end{array}
\]

$(A_{2})$ For all $t_{0}$, $h$ and all bounded $\mathcal{E}_{t_{0}}$-mesurable
random variables $\alpha$, the control process $\beta(t)$ defined by%

\begin{equation}%
\begin{array}
[c]{c}%
\beta(t)=\alpha1_{\left[  s,s+h\right]  }(t)\text{ }%
\end{array}
\tag{$4.1$}%
\end{equation}

belongs to $\mathcal{A}_{\mathcal{E}}$.

$(A_{3})$ For all bounded $\beta\in\mathcal{A}_{\mathcal{E}}$, the derivative process%

\begin{equation}%
\begin{array}
[c]{c}%
\xi(t):=\dfrac{d}{ds}X^{u+s\beta}(t)\mid_{_{s=0}}%
\end{array}
\tag{$4.2$}%
\end{equation}

exists and belongs to $L^{2}(\lambda\times P)$.

It follows from $(2.1)$ that

$%
\begin{array}
[c]{l}%
d\xi(t)=\left\{  \dfrac{\partial b}{\partial x}(t)\xi(t)+\dfrac{\partial
b}{\partial y}(t)\xi(t-\delta)+\dfrac{\partial b}{\partial a}(t)%
{\displaystyle\int\limits_{t-\delta}^{t}}
e^{-\rho(t-r)}\xi(r)dr+\dfrac{\partial b}{\partial u}(t)\beta(t)\right\}  dt\\
+\left\{  \dfrac{\partial\sigma}{\partial x}(t)\xi(t)+\dfrac{\partial\sigma
}{\partial y}(t)\xi(t-\delta)+\dfrac{\partial\sigma}{\partial a}(t)%
{\displaystyle\int\limits_{t-\delta}^{t}}
e^{-\rho(t-r)}\xi(r)dr+\dfrac{\partial\sigma}{\partial u}(t)\beta(t)\right\}
dB(t)\\
+%
{\displaystyle\int\limits_{\mathbb{R} _{0}}}
\left\{  \dfrac{\partial\theta}{\partial x}(t,z)\xi(t)+\dfrac{\partial\theta
}{\partial y}(t,z)\xi(t-\delta)+\dfrac{\partial\theta}{\partial a}(t,z)%
{\displaystyle\int\limits_{t-\delta}^{t}}
e^{-\rho(t-r)}\xi(r)dr+\dfrac{\partial\theta}{\partial u}(t,z)\beta
(t)\right\}  \overset{\sim}{N}(dt,dz)\text{,}%
\end{array}
$

where, for simplicity of notation, we define%

\[%
\begin{array}
[c]{c}%
\dfrac{\partial}{\partial x}b(t):=\dfrac{\partial}{\partial x}%
b(t,X(t),X(t-\delta),A(t),u(t))\text{,}%
\end{array}
\]

and used that%

\[%
\begin{array}
[c]{c}%
\dfrac{d}{ds}Y^{u+s\beta}(t)\mid_{_{s=0}}=\dfrac{d}{ds}X^{u+s\beta}%
(t - \delta)\mid_{_{s=0}}=\xi(t-\delta)
\end{array}
\]

and%

\[%
\begin{array}
[c]{c}%
\dfrac{d}{ds}A^{u+s\beta}(t)\mid_{_{s=0}}=\dfrac{d}{ds}\left(
{\displaystyle\int\limits_{t-\delta}^{t}}
e^{-\rho(t-r)}X^{u+s\beta}(r)dr\right)  \mid_{_{s=0}}\\
\text{ \ \ \ \ \ \ \ \ \ \ \ \ \ \ \ \ \ \ \ \ \ \ \ \ \ }=\left(
{\displaystyle\int\limits_{t-\delta}^{t}}
e^{-\rho(t-r)}\dfrac{d}{ds}X^{u+s\beta}(r)\right)  \mid_{_{s=0}}dt\\
=%
{\displaystyle\int\limits_{t-\delta}^{t}}
e^{-\rho(t-r)}\xi(t)dr\text{.}%
\end{array}
\]

Note that
\[%
\begin{array}
[c]{c}%
\xi(t)=0\text{ for }t\in\left[  -\delta,\infty\right)
\end{array}
\text{.}%
\]

\begin{theorem}[Necessary maximum principle]
Suppose that $\hat{u}\in\mathcal{A}_{\mathcal{E}}$ with corresponding
solutions $\hat{X}(t)$ of $(2.1)$-$(2.2)$ and $\overset{\wedge
}{p}(t)$, $\hat{q}(t)$, and $\hat{r}(t,z)$ of $(3.2)
$-$(3.3)$, and corresponding derivative process $\hat{\xi}(t)$
given by $(4.2)$.

Assume that for all $u\in\mathcal{A}_{\mathcal{E}}$\ the following hold:%
\begin{align*}
&
\begin{array}
[c]{c}%
E\left[
{\displaystyle\int\limits_{0}^{\infty}}
\hat{p^{2}}(t)\left\{  \left(  \dfrac{\partial\sigma}{\partial
x}\right)  ^{2}(t)\hat{\xi}^{2}(t)+\left(  \dfrac{\partial\sigma
}{\partial y}\right)  ^{2}(t)\hat{\xi}^{2}(t-\delta)+\left(
\dfrac{\partial\sigma}{\partial a}\right)  ^{2}(t)\left(
{\displaystyle\int\limits_{t-\delta}^{t}}
e^{-\rho(t-r)}\hat{\xi}(r)dr\right)  ^{2}+\left(  \dfrac
{\partial\sigma}{\partial u}\right)  ^{2}(t)\right.  \right.
\end{array}
\\
&
\begin{array}
[c]{c}%
\left.  +%
{\displaystyle\int\limits_{\mathbb{R} _{0}}}
\left\{  \left(  \dfrac{\partial\theta}{\partial x}\right)  ^{2}%
(t,z)\hat{\xi}^{2}(t)+\left(  \dfrac{\partial\theta}{\partial
y}\right)  ^{2}(t,z)\hat{\xi}^{2}(t-\delta)+\left(  \dfrac
{\partial\theta}{\partial a}\right)  ^{2}(t,z)\left(
{\displaystyle\int\limits_{t-\delta}^{t}}
e^{-\rho(t-r)}\hat{\xi}(r)dr\right)  ^{2}+\left(  \dfrac
{\partial\theta}{\partial u}\right)  ^{2}(t,z)\right\}  \nu(dz)\right\}  dt
\end{array}
\\
&
\begin{array}
[c]{c}%
+%
{\displaystyle\int\limits_{0}^{\infty}}
\hat{\xi}^{2}(t)\left\{  \hat{q}(t)+%
{\displaystyle\int\limits_{\mathbb{R} _{0}}}
\hat{r}^{2}(t,z)\text{ }\nu(dz)\right\}  dt<\infty\text{.}%
\end{array}
\end{align*}

and%

\[%
\begin{array}
[c]{c}%
E\left[  \overline{\underset{t\rightarrow\infty}{\lim}}\mathbf{\ \ }%
\hat{p}(t)(X(t)-\hat{X}(t))\right]  \geq0\text{.}%
\end{array}
\]

Then the following assertions are equivalent.

$(i)$ For all bounded $\beta\in\mathcal{A}_{\mathcal{E}}$,%

\[
\dfrac{d}{ds}J(\hat{u}+s\beta)\mid_{_{s=0}}=0\text{.}%
\]

$(ii)$ For all $t\in\lbrack0$, $\infty)$,%

\[%
\begin{array}
[c]{c}%
E\left[  H(t,\text{ }\hat{X}(t),\hat{Y}%
(t),\hat{A}(t),\hat{u}(t),\hat{p}(t),\overset{\wedge
}{q}(t),\hat{r}(t,.)\mid{\mathcal{E}}_{t}\right]  _{u=\hat{u}%
(t)}=0
\end{array}
\text{a.s.}%
\]

\end{theorem}

\begin{proof}
Suppose that assertion $(i)$ holds. Then%

\[%
\begin{array}
[c]{l}%
0=\dfrac{d}{ds}J(\hat{u}+s\beta)\mid_{_{s=0}}\\
=\dfrac{d}{ds}E\left[
{\displaystyle\int\limits_{0}^{\infty}}
f(t,X^{\hat{u}+s\beta}(t),Y^{\hat{u}+s\beta}(t),A^{\hat{u}+s\beta}(t),\hat
{u}(t)+s\beta(t)dt\right]  _{_{s=0}}\\
=E\left[
{\displaystyle\int\limits_{0}^{\infty}}
\left\{  \dfrac{\partial f}{\partial x}(t)\xi(t)+\dfrac{\partial f}{\partial
y}(t)\xi(t-\delta)+\dfrac{\partial f}{\partial a}(t)%
{\displaystyle\int\limits_{t-\delta}^{t}}
e^{-\rho(t-r)}\xi(r)dr+\dfrac{\partial f}{\partial u}(t)\beta(t)\right\}
dt\right]
\end{array}
\]
$\ \ \ \ $

We know by the definition of $H$ that$\ \ \ \ \ \ \ \ \ \ \ $%
\[%
\begin{array}
[c]{l}%
\dfrac{\partial f}{\partial x}(t)=\dfrac{\partial H}{\partial x}%
(t)-\dfrac{\partial b}{\partial x}(t)p(t)-\dfrac{\partial\sigma}{\partial
x}(t)q(t)-%
{\displaystyle\int\limits_{\mathbb{R}_0}}
\dfrac{\partial\theta}{\partial x}(t,z)r(t,z)\nu(dz)
\end{array}
\]
$\ \ $

and the same for $\dfrac{\partial f}{\partial y}(t),\dfrac{\partial
f}{\partial a}(t)$ and $\dfrac{\partial f}{\partial u}(t).$

We have $\ \ $%
\[%
\begin{array}
[c]{c}%
E\left[  \overline{\underset{t\rightarrow\infty}{\lim}}\mathbf{\ \ }%
\hat{p}(t)(X(t)-\hat{X}(t))\right]  \geq\ 0\ \ \
\end{array}
\
\]
$\ \ \ \ \ \ \ \ \ \ \ $

So $\ \ \ \ \ \ \ \ $%
\[%
\begin{array}
[c]{c}%
E\left[  \overline{\underset{t\rightarrow\infty}{\lim}}\mathbf{\ \ }\left(
\hat{p}(t)X^{\hat{u}+s\beta}(t)\right)  \right]  \geq E\left[
\overline{\underset{t\rightarrow\infty}{\lim}}\mathbf{\ \ }\left(
\hat{p}(t)X^{\hat{u}}(t)\right)  \right]
\end{array}
\ \
\]
$\ \ \ $

for all $\beta\in\mathcal{A}_{\mathcal{E}}$ and all $s\in(-\epsilon,\epsilon)$.

Hence
\[%
\begin{array}
[c]{c}%
\dfrac{d}{ds}\left[  E\left\{  \overline{\underset{t\rightarrow\infty}{\lim}%
}\mathbf{\ \ }\left(  \hat{p}(t)X^{\hat{u}+s\beta}(t)\right)
\right\}  \right]  _{s=0}=0
\end{array}
\
\]

If $%
\begin{array}
[c]{c}%
\dfrac{d}{ds}\left\vert \left(  \overline{\underset{t\rightarrow\infty}{\lim}%
}\mathbf{\ \ }\hat{p}(t)X^{\hat{u}+s\beta}(t)\right)  \right\vert
\mid_{s=0}<g(w)
\end{array}
$, where $g(w)$ is some integrable function. From the uniform limits with
uniform convergence of the derivative, we can interchange the derivative and integration, and get
\[%
\begin{array}
[c]{l}%
0=\dfrac{d}{ds}\left[  E\left\{  \overline{\underset{t\rightarrow\infty}{\lim
}}\mathbf{\ \ }\left(  \hat{p}(t)X^{\hat{u}+s\beta}(t)\right)
\right\}  \right]  \mid_{s=0}\\
=E\left[  \dfrac{d}{ds}\left\{  \overline{\underset{t\rightarrow\infty}{\lim}%
}\mathbf{\ \ }\left(  \hat{p}(t)X^{\hat{u}+s\beta}(t)\right)
\right\}  \right]  \mid_{s=0}\\
=E\left[  \overline{\underset{t\rightarrow\infty}{\lim}}\mathbf{\ }\left\{
\hat{p}(t)\mathbf{\ }\dfrac{d}{ds}\left(  X^{\hat{u}+s\beta
}(t)\right)  \right\}  \right]  \mid_{s=0}\text{.}%
\end{array}
\]

Applying the It\^{o} formula to $%
\begin{array}
[c]{c}%
\hat{p}(t)\mathbf{\ }\dfrac{d}{ds}\left(  X^{\hat{u}+s\beta
}(t)\right)
\end{array}
$ we obtain%
\[%
\begin{array}
[c]{l}%
0=E\left[  \overline{\underset{T\rightarrow\infty}{\lim}}\mathbf{\ }\left\{
\hat{p}(T)\mathbf{\ }\dfrac{d}{ds}\left(  X^{\hat{u}+s\beta
}(T)\right)  \mid_{s=0}\right\}  \right]  =E\left[  \overline
{\underset{T\rightarrow\infty}{\lim}}\mathbf{\ }\left\{  \overset{\wedge
}{p}(T)\xi(T)\right\}  \right] \\
=E\left[
{\displaystyle\int\limits_{0}^{\infty}}
\hat{p}(t)\mathbf{\ }\left\{  \dfrac{\partial b}{\partial x}%
(t)\xi(t)+\dfrac{\partial b}{\partial y}(t)\xi(t-\delta)+\dfrac{\partial
b}{\partial a}(t)%
{\displaystyle\int\limits_{t-\delta}^{t}}
e^{-\rho(t-r)}\xi(r)dr+\dfrac{\partial b}{\partial u}(t)\beta(t)\right\}
dt\right. \\
+%
{\displaystyle\int\limits_{0}^{\infty}}
\xi(t)E\left(  \mu(t)\mid\mathcal{F}_{t}\right)  dt+%
{\displaystyle\int\limits_{0}^{\infty}}
q(t)\left\{  \dfrac{\partial\sigma}{\partial x}(t)\xi(t)+\dfrac{\partial
\sigma}{\partial y}(t)\xi(t-\delta)+\dfrac{\partial\sigma}{\partial a}(t)%
{\displaystyle\int\limits_{t-\delta}^{t}}
e^{-\rho(t-r)}\xi(r)dr+\dfrac{\partial\sigma}{\partial u}(t)\beta(t)\right\}
dt\\
\left.  +%
{\displaystyle\int\limits_{0}^{\infty}}
{\displaystyle\int\limits_{\mathbb{R}_0}}
r(t,z)\left\{  \dfrac{\partial\theta}{\partial x}(t,z)\xi(t)+\dfrac
{\partial\theta}{\partial y}(t,z)\xi(t-\delta)+\dfrac{\partial\theta}{\partial
a}(t,z)%
{\displaystyle\int\limits_{t-\delta}^{t}}
e^{-\rho(t-r)}\xi(r)dr+\dfrac{\partial\theta}{\partial u}(t,z)\beta
(t)\right\}  \nu(dz)dt\right] \\
=-\dfrac{d}{ds}J(\hat{u}+s\beta)\mid_{s=0}+E\left(
{\displaystyle\int\limits_{0}^{\infty}}
\dfrac{\partial H}{\partial u}(t)\beta(t)dt\right)  \text{.}%
\end{array}
\]

Therefore%
\[%
\begin{array}
[c]{c}%
E\left(
{\displaystyle\int\limits_{0}^{\infty}}
\dfrac{\partial H}{\partial u}(t)\beta(t)dt\right)  =0
\end{array}
\text{.}%
\]

Use%
\[%
\begin{array}
[c]{c}%
\beta(t)=\alpha1_{\left[  s,s+h\right]  }(t)\text{ }%
\end{array}
\]

where $\alpha(\omega)$ is bounded and $\mathcal{E}_{t_{0}}$-mesurable, $s\geq
t_{0}$ and get%
\[%
\begin{array}
[c]{c}%
E\left(
{\displaystyle\int\limits_{s}^{s+h}}
\dfrac{\partial H}{\partial u}(s)ds\text{ }\alpha\right)  =0
\end{array}
\]

Differentiating with respect to $h$ at $0$, we have%
\[%
\begin{array}
[c]{c}%
E\left(  \dfrac{\partial H}{\partial u}(s)\text{ }\alpha\right)  =0
\end{array}
\]

This holds for all $s\geq t_{0}$ and all $\alpha$, we obtain that%
\[%
\begin{array}
[c]{c}%
E\left(  \dfrac{\partial H}{\partial u}(t_{0})\mid\mathcal{E}_{t_{0}}\right)
=0
\end{array}
\text{.}%
\]

This proves that assertion $(i)$ implies $(ii)$.

To complete the proof, we need to prove the converse implication; which is obtained
since every bounded $\beta\in\mathcal{A}_{\mathcal{E}}$ can be
approximated by linear combinations of controls $\beta$ of the form $(4.1)$.
\end{proof}

\section{Existence and uniqueness of the time-advanced BSDEs on infinite
horizon}

The main result in this section refer to the existence and uniqueness for
$(3.3)-(3.4)$ where the coefficients satisfy a Lipschitz condition.

We now study time-advanced backward stochastic  differential equations driven by a Brownian motion $B(t)$, and a compensated Poisson random measure $\tilde{N}(dt,d\zeta)$.

Let $B(t)$ be a Brownian motion and $\tilde{N}(dt,d\zeta) := N(dt,d\zeta) - \nu(d\zeta)dt$, where $\nu$ is the L\'{e}vy measure of the jump meaure $N(\cdot,\cdot)$, be an independent compensated Poisson random measure on a filtered probability space $(\Omega,\mathcal{F},\{\mathcal{F}_t \}_{0\leq t < \infty})$.

Given a positive constant $\delta$, denote by $D([0,\delta],\mathbb{R})$ the space of all càdlàg paths from $[0,\delta]$ into $\mathbb{R}$. 
For a path $X(\cdot): \mathbb{R}_{+} \to \mathbb{R}$, $X_t$ will denote the function defined by $X_t(s) = X(t+s)$ for $s \in [0,\delta]$. Put $\mathcal{H} = L^2(\nu)$. Consider the $L^2$ space $V_1 := L^2([0,\delta] \to \mathbb{R};ds)$ and $V_2 := L^2([0,\delta] \to \mathcal{H};ds)$.
Let 
\begin{align*}
 F: \mathbb{R}_{+} \times \mathbb{R} \times \mathbb{R} \times V_1 \times \mathbb{R} \times \mathbb{R} \times V_1 \times \mathcal{H} \times \mathcal{H} \times V_2 \times \Omega \to \mathbb{R}
\end{align*}
be a function satisfying the following Lipschitz condition: There exists a constant $C$ such that
\begin{align*}
&|F(t,p_1,p_2,p,q_1,q_2,q,r_1,r_2,r,\omega) - F(t,\bar{p}_1,\bar{p_2},\bar{p},\bar{q}_1,\bar{q}_2,\bar{q},\bar{r}_1,\bar{r}_2,\bar{r},\omega)|\\
&\leq C( |p_1 - \bar{p}_1| + |p_2 - \bar{p}_2| + |p - \bar{p}|_{V_1} + |q_1 - \bar{q}_1| + |q_2 - \bar{q}_2| + |q - \bar{q}|_{V_1} \notag\\
&+ |r_1 - \bar{r}_1|_{\mathcal{H}}+ |r_2 - \bar{2}_2|_{\mathcal{H}} + |r - \bar{r}|_{V_2}  ). \tag{5.1}
\end{align*}
Assume that $(t,\omega) \to F(t,p_1,p_2,p,q_1,q_2,q,r_1,r_2,r,\omega)$ is predictable for all $p_1,p_2,p,q_1,q_2,q,r_1,r_2,r$.
Further we assume the following:
\begin{align*}
E \int_0^{\infty} e^{\lambda t} |F(t,0,0,0,0,0,0,0,0,0)|^2dt < \infty,
\end{align*}
for all $\lambda \in \mathbb{R}$.
We now consider the following backward stochastic differential equation in the unknown $\mathcal{F}_t$-adapted processes $(p(t),q(t),r(t,z)) \in H \times H \times \mathcal{H}$:
\begin{align*} 
dp(t) &= E[F(t,p(t),p(t+\delta,p_t,q(t+\delta),q_t,r(t+\delta),r_t)|\mathcal{F}_t]dt\\ 
&+q(t)dB(t) + \int_{\mathbb{R}_0}r(t,z)\tilde{N}(dz,dt), \tag{5.2}
\end{align*}
where
\begin{align*} 
 E\left[ \int_0^{\infty} e^{\lambda t} |p(t)|^2  dt \right] < \infty, \tag{5.3}
\end{align*} 
for all $\lambda \in \mathbb{R}$.

\begin{theorem}[Existence and uniqueness]
Assume the condition $(5.1)$ is fulfilled.  Then the backward stochastic partial differential equation $(5.2)$ - $(5.3)$ admits a unique solution $(p(t),q(t),r(t,z))$ such that
\begin{align*}
 E\left[\int_0^{\infty} e^{\lambda t}\{ |p(t)|^2 + |q(t)|^2 + \int_{\mathbb{R}_0}|r(t,z)|^2\nu(dz) \} dt \right]< \infty, 
\end{align*}
for all $\lambda \in \mathbb{R}$.
\end{theorem}

\begin{proof}\newline
\textbf{Step 1:}\newline
Assume F is independent of its second, third and fourth parameter. \newline

Set $q^0(t) :=0$, $r^0(t,z):=0$. For $n \geq 1$, define 
$(p^n(t),q^n(t),r^n(t,z))$  to be the unique solution of the following BSDE:
\begin{align*} 
dp^{n}(t) &=  E\left[F(t,q^{n-1}(t),q^{n-1}(t+\delta),q_{t}^{n-1},r^{n-1}(t,.),r^{n-1}(t+\delta,.),r_{t}^{n-1}(\cdot)) \mid\mathcal{F}_{t}\right]dt \notag\\
&+q^{n}(t)dB(t) +\int_{\mathbb{R}_0} r^{n}(t,z) \tilde{N}(dt,dz); \tag{5.4}
\end{align*}
for $t\in\left[0,\infty\right)$ such that 
\begin{align*}
E\left[ \int_{0}^{\infty} e^{\lambda t} \lvert p^{n}(t) \rvert^{2} dt\right]  <\infty
\end{align*}
This exists by Theorem 3.1 in \cite{HOP}.

Our goal is to show that $(p^n(t),q^n(t),r^n(t,z))$ forms a Cauchy sequence.

By It$\bar{o}$'s formula  we get that
\begin{align*}
0 &= \lvert e^{\lambda t} p^{n+1}(t)-p^{n}(t) \rvert^{2} 
+ \int_{t}^{\infty}  \lambda e^{\lambda s} \lvert p^{n+1}(s)- p^{n}(s) \rvert^{2}ds\\
&+ \int_{t}^{\infty}  e^{\lambda s} \lvert q^{n+1}(s)-q^{n}(s) \rvert^{2}ds\\
&+ \int_{t}^{\infty} e^{\lambda s}  \int_{\mathbb{R}_0} \lvert \left(  r^{n+1}(s,z)-r^{n}(s,z)\right)  \rvert^{2}ds \nu(dz)\\
&+ 2\int_{t}^{\infty}  e^{\lambda s} \langle \left(  p^{n+1}(s)-p^{n}(s)\right), E\left[  F^n - F^{n-1} \mid\mathcal{F}_{s} \right] \rangle ds\\
&+ 2 \int_{t}^{\infty} e^{\lambda s} \left( p^{n+1}(s)-p^{n}(s)\right)  (q^{n+1}(s)-q^{n}(s))dB_{s}\\
&+\int_{t}^{\infty}\int_{\mathbb{R}_0}e^{\lambda s}( |r^{n+1}(s,z)-r^{n}(s,z)|^2\\
&+   2 \left( p^{n+1}(s^{-})-p^{n}(s^{-}) \right) \left(  r^{n+1}(s,z)-r^{n}(s,z)\right) )  \tilde{N}(ds, dz).
\end{align*}

Rearenging, using that for all $\epsilon > 0$, $ab \leq \frac{a^2}{\epsilon} + \epsilon b^2$ we have by the Lipschitz requirement $(5.1)$ 
\begin{align*}
&E[ e^{\lambda t} \lvert   p^{n+1}(t)-p^{n}(t) \rvert^{2} ] \\
&+ \int_{t}^{\infty}  \lambda e^{\lambda s} \lvert p^{n+1}(s)- p^{n}(s) \rvert^{2}ds\\
&+ E[\int_{t}^{\infty}  e^{\lambda s} \lvert q^{n+1}(s)-q^{n}(s) \rvert^{2}ds]\\
&+ E[\int_{t}^{\infty} \int_{\mathbb{R}_0} e^{\lambda s} \lvert \left(  r^{n+1}(s,z)-r^{n}(s,z)\right)  \rvert^{2} \nu(dz)ds]\\
&\leq C_{\epsilon} E[ \int_{t}^{\infty} e^{\lambda s} \lvert p^{n+1}(s)-p^{n}(s) \rvert^{2}ds \\
&+ \epsilon 6 E[\int_{t}^{\infty} e^{\lambda s} \lvert q^{n}(s)-q^{n-1}(s) \rvert^{2}ds]\\
&+ \epsilon 6 E[\int_{t}^{\infty} e^{\lambda s} \lvert q^{n}(s+ \delta)-q^{n-1}(s+\delta) \rvert^{2}ds]\\
&+ \epsilon 6 E[\int_{t}^{\infty} e^{\lambda s} \int_s^{s+\delta} \lvert q^{n}(u)-q^{n-1}(u) \rvert^{2}duds]\\
&+ \epsilon 6 E[\int_{t}^{\infty} e^{\lambda s} \lvert r^{n}(s)-r^{n-1}(s) \rvert_{\mathcal{H}}^{2}ds]\\
&+ \epsilon 6 E[\int_{t}^{\infty} e^{\lambda s} \lvert r^{n}(s+ \delta)-r^{n-1}(s+\delta) \rvert_{\mathcal{H}}^{2}ds]\\
&+ \epsilon 6 E[\int_{t}^{\infty} e^{\lambda s} \int_s^{s+\delta} \lvert r^{n}(u)-r^{n-1}(u) \rvert_{\mathcal{H}}^{2}duds]\\
\end{align*}
where $C_{\epsilon} = \frac{C^2}{\epsilon}$ and we used the abbreviation
\[
F^n(t) := F(t,q^{n}(t),q^{n}(t+\delta),q_{t}^{n},r^{n}(t,.),r^{n}(t+\delta,.),r_{t}^{n}(\cdot)).
\]
 
Note that
\begin{align*}
E[&\int_{t}^{\infty} e^{\lambda s} \lvert q^{n}(s+ \delta)-q^{n-1}(s+\delta) \rvert^{2}ds\\
&\leq   e^{-\lambda \delta} E[\int_{t}^{\infty} e^{\lambda s} \lvert q^{n}(s)-q^{n-1}(s) \rvert^{2}ds].
\end{align*}
Using Fubini
\begin{align*}
&E[\int_{t}^{\infty} \int_s^{s+\delta} e^{\lambda s} \lvert q^{n}(u)-q^{n-1}(u) \rvert^{2}duds\\
&\leq E[\int_{t}^{\infty}  \int_{u-\delta}^{u} e^{\lambda s} \lvert q^{n}(u)-q^{n-1}(u) \rvert^{2}dsdu \\
&\leq (\frac{1}{\lambda}(1- e^{-\lambda \delta})E[\int_{t}^{\infty} e^{\lambda s} \lvert q^{n}(s)-q^{n-1}(s) \rvert^{2}ds]\\
&\leq E[\int_{t}^{\infty} e^{\lambda s} \lvert q^{n}(s)-q^{n-1}(s) \rvert^{2}ds].
\end{align*}
Similiar for $r^n - r^{n-1}$. It now follows that
\begin{align*} 
&Ee^{\lambda t} \lvert p^{n+1}(t)-p^{n}(t) \rvert^{2} ]\notag\\
&+ E[\int_{t}^{\infty} e^{\lambda s} \lvert q^{n+1}(s)-q^{n}(s) \rvert^{2}ds] \notag\\
&+ E[\int_{t}^{\infty} \int_{\mathbb{R}_0} e^{\lambda s} \lvert \left(  r^{n+1}(s,z)-r^{n}(s,z)\right)  \rvert^{2} \nu(dz)ds] \notag\\
&\leq (C_{\epsilon} - \lambda)E[ \int_{t}^{\infty} e^{\lambda s} \lvert p^{n+1}(s)-p^{n}(s) \rvert^{2}ds]\notag\\ 
&+ \epsilon 6(2 + e^{-\lambda \delta}) E[\int_{t}^{\infty} e^{\lambda s} \lvert q^{n}(s)-q^{n-1}(s) \rvert^{2}ds] \notag\\
& + \epsilon  6(2 + e^{-\lambda \delta}) E[\int_{t}^{\infty} e^{\lambda s} \lvert r^{n}(s)-r^{n-1}(s) \rvert_{\mathcal{H}}^{2}ds]. \tag{5.5}
\end{align*}
Choosing $\epsilon = \frac{1}{12(2+e^{-\lambda \delta})}$ so that
\begin{align*}
&E e^{\lambda t} \lvert p^{n+1}(t)-p^{n}(t) \rvert^{2} ] \\
&+ E[\int_{t}^{\infty} e^{\lambda s} \lvert q^{n+1}(s)-q^{n}(s) \rvert^{2}ds]\\
&+ E[\int_{t}^{\infty} \int_{\mathbb{R}_0} e^{\lambda s} \lvert \left(  r^{n+1}(s,z)-r^{n}(s,z)\right)  \rvert^{2} \nu(dz)ds]\\
&\leq (C_{\epsilon} - \lambda)E[ \int_{t}^{\infty} e^{\lambda s} \lvert p^{n+1}(s)-p^{n}(s) \rvert^{2}ds] 
+ \frac{1}{2}E[\int_{t}^{\infty} e^{\lambda s} \lvert q^{n}(s)-q^{n-1}(s) \rvert^{2}ds]\\
& + \frac{1}{2} E[\int_{t}^{\infty} e^{\lambda s} \lvert r^{n}(s)-r^{n-1}(s) \rvert_{\mathcal{H}}^{2}ds]\\
\end{align*}
This implies that
\begin{align*}
&- \frac{\partial }{\partial t}(e^{(C_{\epsilon} - \lambda) t}E[ \int_0^{\infty}  e^{\lambda s} \lvert p^{n+1}(t)-p^{n}(t) \rvert^{2} ]\\
&+ e^{(C_{\epsilon} - \lambda) t}E[\int_{t}^{\infty} e^{\lambda s} \lvert q^{n+1}(s)-q^{n}(s) \rvert^{2}]ds\\
&+ e^{(C_{\epsilon} - \lambda) t}E[\int_{t}^{\infty} \int_{\mathbb{R}_0} e^{\lambda s} \lvert \left(  r^{n+1}(s,z)-r^{n}(s,z)\right)  \rvert^{2} \nu(dz)ds]\\
&\leq \frac{1}{2}e^{(C_{\epsilon} - \lambda) t} E[\int_{t}^{\infty} e^{\lambda s} \lvert q^{n}(s)-q^{n-1}(s) \rvert^{2}ds\\
& + \frac{1}{2}e^{(C_{\epsilon} - \lambda) t} E\int_{t}^{\infty} e^{\lambda s} \lvert r^{n}(s)-r^{n-1}(s) \rvert_{\mathcal{H}}^{2}ds\\
\end{align*}
Integrating the last inequality we get that
\begin{align*} 
& E[ \int_0^{\infty} e^{\lambda s} \lvert p^{n+1}(t)-p^{n}(t) \rvert^{2} dt]\\
&+ \int_0^{\infty}e^{(C_{\epsilon} - \lambda) t}E[\int_{t}^{\infty} e^{\lambda s} \lvert q^{n+1}(s)-q^{n}(s) \rvert^{2}ds] \notag\\
&+ \int_0^{\infty}e^{(C_{\epsilon} - \lambda) t}E[\int_{t}^{\infty} \int_{\mathbb{R}_0} e^{\lambda s} \lvert \left(  r^{n+1}(s,z)-r^{n}(s,z)\right)  \rvert^{2} \nu(dz)ds]dt \notag\\
&\leq \frac{1}{2}\int_0^{\infty}e^{(C_{\epsilon} - \lambda) t} E[\int_{t}^{\infty} e^{\lambda s} \lvert q^{n}(s)-q^{n-1}(s) \rvert^{2}dsdt \notag\\
& + \frac{1}{2}\int_0^{\infty}e^{(C_{\epsilon} - \lambda) t} E\int_{t}^{\infty} e^{\lambda s} \lvert r^{n}(s)-r^{n-1}(s) \rvert_{\mathcal{H}}^{2}dsdt. \tag{5.6}
\end{align*}
So that
\begin{align*}
&\int_0^{\infty}e^{(C_{\epsilon} - \lambda) t}E[\int_{t}^{\infty} e^{\lambda s} \lvert q^{n+1}(s)-q^{n}(s) \rvert^{2}ds]\\
&+ \int_0^{\infty}e^{(C_{\epsilon} - \lambda) t}E[\int_{t}^{\infty} \int_{\mathbb{R}_0} e^{\lambda s} \lvert \left(  r^{n+1}(s,z)-r^{n}(s,z)\right)  \rvert^{2} \nu(dz)ds]dt\\
&\leq \frac{1}{2}\int_0^{\infty}e^{(C_{\epsilon} - \lambda) t} E[\int_{t}^{\infty} e^{\lambda s} \lvert q^{n}(s)-q^{n-1}(s) \rvert^{2}dsdt\\
& + \frac{1}{2}\int_0^{\infty}e^{(C_{\epsilon} - \lambda) t} E\int_{t}^{\infty} e^{\lambda s} \lvert r^{n}(s)-r^{n-1}(s) \rvert_{\mathcal{H}}^{2}dsdt\\
\end{align*}
This gives that
\begin{align*}
&\int_0^{\infty}e^{(C_{\epsilon} - \lambda) t}E[\int_{t}^{\infty} e^{\lambda s} \lvert q^{n+1}(s)-q^{n}(s) \rvert^{2}ds]\\
&+ \int_0^{\infty}e^{(C_{\epsilon} - \lambda) t}E[\int_{t}^{\infty} \int_{\mathbb{R}_0} e^{\lambda s} \lvert \left(  r^{n+1}(s,z)-r^{n}(s,z)\right)  \rvert^{2} \nu(dz)ds]dt\\
&\leq \frac{1}{2^n}C_3,
\end{align*}
if  $\lambda > \frac{C}{\epsilon}$.
It then follows from $(5.6)$ that 
\[
E[ \int_0^{\infty}  e^{\lambda t} \lvert p^{n+1}(t)-p^{n}(t) \rvert^{2} ] \leq \frac{1}{2^n}C_3.
\]
From $(5.5)$ and $(5.6)$, we now get
\begin{align*}
&E[\int_{t}^{\infty} \int_{\mathbb{R}_0} e^{\lambda s} \lvert \left(  r^{n+1}(s,z)-r^{n}(s,z)\right)  \rvert^{2} \nu(dz)ds]dt\\
&+ E[\int_{t}^{\infty} e^{\lambda s}\lvert q^{n+1}(s)-q^{n}(s) \rvert^{2}ds] \leq \frac{1}{2^n}C_3nC_{\epsilon}.
\end{align*}

From this we conclude that there exist progressively measurable processes $(p(t),q(t),r(t,z))$, such that
\begin{align*}
\underset{n \to \infty}{\lim} E[ e^{\lambda t} \lvert p^{n}(t)-p(t) \rvert^{2} dt] &= 0,\\
\underset{n \to \infty}{\lim} E[ \int_0^{\infty}  e^{\lambda s}\lvert p^{n}(t)-p(t) \rvert^{2} dt] &= 0,\\
\underset{n \to \infty}{\lim} E[ \int_0^{\infty}  e^{\lambda s}\lvert p^{n}(t)-p(t) \rvert^{2} dt] &= 0,\\
\underset{n \to \infty}{\lim} E[ \int_0^{\infty}  e^{\lambda s}\lvert p^{n}(t)-p(t) \rvert^{2} dt] &= 0,\\
\underset{n \to \infty}{\lim} E[\int_{t}^{\infty} \int_{\mathbb{R}_0} e^{\lambda s}\lvert \left(  r^{n}(s,z)-r(s,z)\right)  \rvert^{2} \nu(dz)ds]dt &= 0.
\end{align*}
Letting $n \to \infty$ in $(5.4)$ we see that $(p(t),q(t),r(t,z))$ satisfies
\begin{align*}
dp(t) &= E\left[  F(t,q(t),q(t+\delta),q_{t},r(t,.),r(t+\delta,.),r_{t}(\cdot)) \mid\mathcal{F}_{t}\right]  dt\\
&+q(t)dB(t)
+\int_{\mathbb{R}_0} r(t,z) \tilde{N}(dt,dz),
\end{align*}
for all $t>0$.\newline 

\textbf{Step 2:}\newline
General F.\newline

Let $p^0(t) =0$. For $n \geq 1$ define $(p^n(t),q^n(t),r^n(t,z))$ to be the unique solution of the following BSDE:
\begin{align*} 
dp^n(t) &= E[F(t,p^{n-1}(t),p^{n-1}(t+\delta),p^{n-1}_t,q^n(t),q^n(t+\delta),q^n_t,r^n(t),r^n(t+\delta),r^n_t)|\mathcal{F}_t]dt\\
&+q^n(t)dB(t) + \int_{\mathbb{R}_0}r^n(t,z)\tilde{N}(dz,dt), 
\end{align*}
for $t \in [0,\infty)$.
The existence of $(p^n(t),q^n(t),r^n(t,z))$ was proved in Step 1. 

By using  the same arguments as above, we deduce that
\begin{align*}
&E  e^{\lambda t} \lvert p^{n+1}(t)-p^{n}(t) \rvert^{2} ] \\
&+ E[\int_{t}^{\infty} e^{\lambda s} \lvert q^{n+1}(s)-q^{n}(s) \rvert^{2}ds]\\
&+ E[\int_{t}^{\infty} e^{\lambda s} \int_{\mathbb{R}_0} \lvert \left(  r^{n+1}(s,z)-r^{n}(s,z)\right)  \rvert^{2} \nu(dz)ds]\\
&\leq (C_{\epsilon} - \lambda) E[ \int_{t}^{\infty}  e^{\lambda s}\lvert p^{n+1}(s)-p^{n}(s) \rvert^{2}ds] 
+ \frac{1}{2}E[\int_{t}^{\infty} e^{\lambda s} \lvert p^{n}(s)-p^{n-1}(s) \rvert^{2}ds]\\
\end{align*}
This implies that
\begin{align*}
 -\frac{d}{dt} (e^{(C_{\epsilon} - \lambda)  t} E[ \int_t^{\infty} e^{\lambda s} |p^{n+1}(s) - p^{n}(s)|^2 ds ] ) 
\leq  \frac{1}{2} e^{(C_{\epsilon} - \lambda)  t} E[ \int_t^{\infty} e^{\lambda s} |p^{n}(s) - p^{n-1}(s)|^2 ds ].
\end{align*}
Integrating from $0$ to $\infty$, we get
\begin{align*}
E[ \int_0^{\infty} e^{\lambda s} |p^{n+1}(s) - p^{n}(s)|^2 ds ] 
&\leq  \frac{1}{2} \int_0^{\infty} e^{(C_{\epsilon} - \lambda)  t} E[ \int_t^{\infty} e^{\lambda s}|p^{n}(s) - p^{n-1}(s)|^2 ds ]dt
\end{align*}
So if $\lambda \geq C_{\epsilon}$ then by iteration we see that
\begin{align*}
E[ \int_0^{\infty} e^{\lambda s}|p^{n+1}(s) - p^{n}(s)|^2 ds ] 
&\leq \frac{K}{2^n(\lambda - C_{\epsilon})^n},
\end{align*}
for some constant $K$.\newline
\textbf{Uniqueness:} \newline
In order to prove the uniqueness, we assume that there
are two solutions $\left( p^{1}(t),q^{1}(s),r^{1}(s,z)\right) $ and $\left(
p^{2}(t),q^{2}(s),r^{2}(s,z)\right) $ of the ABSDE%
\begin{eqnarray*}
&&%
\begin{array}{c}
dp(t)=E\left[ F(t,p(t),p(t+\delta ),p_{t},q(t),q(t+\delta
),q_{t},r(t),r(t+\delta ),r_{t})\mid \mathcal{F}_{t}\right] dt \\ 
+q(t)dB(t)+\int\limits_{%
\mathbb{R}
_{0}}r(t,z)\overset{\sim }{N}(dt,dz);\text{ \ \ \ \ \ \ \ \ \ \ \ \ \ \ \ }%
t\in \left[ 0,\infty \right)%
\end{array}
\\
&&%
\begin{array}{c}
E\left[ \int\limits_{0}^{\infty }\text{ }e^{\lambda \text{ }t}\text{ }%
\left\vert p(t)\right\vert ^{2}\text{ }dt\right] <\infty \text{ ; }\lambda
\in 
\mathbb{R}%
\end{array}%
\text{.}
\end{eqnarray*}

By It$\bar{o}$'s formula, we have

\begin{equation*}
\begin{array}{l}
E\left[ e^{\lambda \text{ }t}\left\vert p^{1}(t)-p^{2}(t)\right\vert ^{2}%
\right] +E\left[ \int\limits_{t}^{\infty }\lambda \text{ }e^{\lambda \text{ 
}s}\left\vert p^{1}(s)-p^{2}(s)\right\vert ds\right] \\ 
+E\left[ \int\limits_{t}^{\infty }e^{\lambda \text{ }s}\left\vert
q^{1}(s)-q^{2}(s)\right\vert ^{2}ds\right] +E\left[ \int\limits_{t}^{\infty
}e^{\lambda \text{ }s}\int\limits_{%
\mathbb{R}
_{0}}\left\vert r^{1}(s,z)-r^{2}(s,z)\right\vert ^{2}ds\nu (dz)\right] \\ 
=2E\int\limits_{t}^{\infty }e^{\lambda \text{ }s}\left[ 
\begin{array}{c}
\left\vert p^{1}(s)-p^{2}(s)\right\vert%
\end{array}%
\right. \\ 
\times \left( E\left[ 
\begin{array}{c}
F(s,p^{1}(s),p^{1}(s+\delta ),p_{s}^{1},q^{1}(s),q^{1}(s+\delta
),q_{s}^{1},r^{1}(s),r^{1}(s+\delta ),r_{s}^{1})\mid \mathcal{F}_{s}%
\end{array}%
\right] \right. \\ 
\\ 
\left. \left. -E\left[ 
\begin{array}{c}
F(s,p^{2}(s),p^{2}(s+\delta ),p_{s}^{2},q^{2}(s),q^{2}(s+\delta
),q_{s}^{2},r^{2}(s),r^{2}(s+\delta ),r_{s}^{2})\mid \mathcal{F}_{s}%
\end{array}%
\right] \right) \right] ds%
\end{array}%
\end{equation*}%
\begin{equation*}
\begin{array}{l}
\leq 2E\int\limits_{t}^{\infty }%
\begin{array}{c}
e^{\lambda \text{ }s}\left[ \left\vert p^{1}(s)-p^{2}(s)\right\vert \right.%
\end{array}
\\ 
\times C\left( \left\vert p^{1}(s)-p^{2}(s)\right\vert +\left\vert
p^{1}(s+\delta )-p^{2}(s+\delta )\right\vert +\int\limits_{s}^{s+\delta
}\left\vert p^{1}(u)-p^{2}(u)\right\vert du\right. \\ 
\begin{array}{c}
+\left\vert q^{1}(s)-q^{2}(s)\right\vert +\left\vert q^{1}(s+\delta
)-q^{2}(s+\delta )\right\vert +\int\limits_{s}^{s+\delta }\left\vert
q^{1}(u)-q^{2}(u)\right\vert du%
\end{array}
\\ 
\left. \left. +\left\vert r^{1}(s)-r^{2}(s)\right\vert _{\mathcal{H}%
}^{2}+\left\vert r^{1}(s+\delta )-r^{2}(s+\delta )\right\vert _{\mathcal{H}%
}^{2}+\int\limits_{s}^{s+\delta }\left\vert r^{1}(u)-r^{2}(u)\right\vert _{%
\mathcal{H}}^{2}du\right) \right] ds%
\end{array}%
\end{equation*}

By the above inequalities for $(p,q,r)$ and the fact that $%
\begin{array}{c}
2ab\leq \dfrac{a^{2}}{\epsilon }+\epsilon b^{2}%
\end{array}%
$, we have that
\begin{equation*}
\begin{array}{l}
E\left[ e^{\lambda \text{ }t}\left\vert p^{1}(t)-p^{2}(t)\right\vert ^{2}%
\right] +E\left[ \int\limits_{t}^{\infty }e^{\lambda \text{ }s}\left\vert
q^{1}(s)-q^{2}(s)\right\vert ^{2}ds\right] \\ 
+E\left[ \int\limits_{t}^{\infty }e^{\lambda \text{ }s}\int\limits_{%
\mathbb{R}
_{0}}\left\vert r^{1}(s,z)-r^{2}(s,z)\right\vert ^{2}ds\nu (dz)\right] \\ 
\leq (\dfrac{3C^{2}}{\epsilon }-\lambda \text{ })\text{ }E\left[
\int\limits_{t}^{\infty }e^{\lambda \text{ }s}\left\vert
p^{1}(s)-p^{2}(s)\right\vert ^{2}ds\right] \\ 
+(2+e^{-\lambda \text{ }%
\delta })\text{ }\epsilon E\left[ \int\limits_{t}^{\infty }e^{\lambda \text{
}s}\left\vert p^{1}(s)-p^{2}(s)\right\vert ^{2}ds\right] \\ 
+(2+e^{-\lambda \text{ }%
\delta })\text{ }\epsilon \text{ }E\int\limits_{t}^{\infty }e^{\lambda 
\text{ }s}\left[ \left\vert q^{1}(s)-q^{2}(s)\right\vert ^{2}+\left\vert
r^{1}(s,z)-r^{2}(s,z)\right\vert _{\mathcal{H}}^{2}\right] ds%
\end{array}%
\end{equation*}

Taking $\epsilon $ such that $%
\begin{array}{c}
(2+e^{-\lambda \text{ }%
\delta })\text{ }\epsilon =\dfrac{1}{2}%
\end{array}%
$%
\begin{equation*}
\begin{array}{l}
E\left[ e^{\lambda \text{ }t}\left\vert p^{1}(t)-p^{2}(t)\right\vert ^{2}%
\right] +E\left[ \int\limits_{t}^{\infty }e^{\lambda \text{ }s}\left\vert
q^{1}(s)-q^{2}(s)\right\vert ^{2}ds\right] \\ 
+E\left[ \int\limits_{t}^{\infty }e^{\lambda \text{ }s}\int\limits_{%
\mathbb{R}
_{0}}\left\vert r^{1}(s,z)-r^{2}(s,z)\right\vert ^{2}ds\nu (dz)\right] \\ 
\leq (\dfrac{3C^{2}}{\epsilon }-\lambda +\dfrac{1}{2}\text{ })E\left[
\int\limits_{t}^{\infty }e^{\lambda \text{ }s}\left\vert
p^{1}(s)-p^{2}(s)\right\vert ^{2}ds\right] \\ 
+\dfrac{1}{2}E\left[ \int\limits_{t}^{\infty }\left\vert
q^{1}(s)-q^{2}(s)\right\vert ^{2}ds\right] \\ 
+\dfrac{1}{2}E\left[ \int\limits_{t}^{\infty }\left\vert
r^{1}(s,z)-r^{2}(s,z)\right\vert _{\mathcal{H}}^{2}ds\right]%
\end{array}%
\end{equation*}

We get 
\begin{equation*}
\begin{array}{l}
E\left[ e^{\lambda \text{ }t}\left\vert p^{1}(t)-p^{2}(t)\right\vert ^{2}%
\right] +\dfrac{1}{2}E\left[ e^{\lambda \text{ }s}\left\vert
q^{1}(s)-q^{2}(s)\right\vert ^{2}ds\right] \\ 
+\dfrac{1}{2}E\left[ \int\limits_{t}^{\infty }e^{\lambda \text{ }%
s}\int\limits_{%
\mathbb{R}
_{0}}\left\vert r^{1}(s,z)-r^{2}(s,z)\right\vert ^{2}ds\nu (dz)\right] \\ 
\leq (\dfrac{3C^{2}}{\epsilon }-\lambda +\dfrac{1}{2}\text{ })E\left[
\int\limits_{t}^{\infty }e^{\lambda \text{ }s}\left\vert
p^{1}(s)-p^{2}(s)\right\vert ^{2}ds\right] \text{.}%
\end{array}%
\end{equation*}

Using the fact that $%
\begin{array}{c}
\lambda \geq \dfrac{3C^{2}}{\epsilon }+\dfrac{1}{2}%
\end{array}%
$, we obtain for all $t\in \left[ 0,\infty \right) $,
\begin{equation*}
\begin{array}{c}
E\left[ e^{\lambda \text{ }t}\left\vert p^{1}(t)-p^{2}(t)\right\vert ^{2}%
\right] =0\text{,}%
\end{array}%
\end{equation*}
which proves that $p^{1}(t)$ and $p^{2}(t)$ are indistinguishable.
\end{proof}

\begin{acknowledgment}
 We thank Salah-E.A. Mohammed for fruitful conversations.
\end{acknowledgment}

\bibliographystyle{plain}
\bibliography{references}

\begin{thebibliography}{10}

\bibitem{BBP}
G.~Barles, R.~Buckdahn, and E.~Pardoux.
\newblock Backward stochastic differential equations and integral-partial
  differential equations.
\newblock {\em Stochastics and Stochastics Reports}, 60:57--83, 2009.

\bibitem{CW}
L.~Chen and Z.~Wu.
\newblock Maximum principle for the stochastic optimal control problem with
  delay and application.
\newblock {\em Automatica}, 46:1074--1080, 2010.

\bibitem{EOS}
I.~Elsanosi, B.~{\O}ksendal, and A.~Sulem.
\newblock Some solvable stochastic control problems with delay.
\newblock {\em Stochastics and Stochastics Reports}, 71:69--89, 2000.

\bibitem{HOP}
S.~Haadem, B.~{\O}ksendal, and F.~Proske.
\newblock Maximum principles for jump diffusion processes with infinite
  horizon.
\newblock {\em eprint arXiv:1206.1719}, 06, 2012.

\bibitem{halkin}
H.~Halkin.
\newblock Necessary conditions for optimal control problems with infinite
  horizons.
\newblock {\em Econometrica}, 42:267--272, 1974.

\bibitem{LP}
J.~Li and S.~Peng.
\newblock Stochastic optimization theory of backward stochastic differential
  equations with jumps and viscosity solutions of {H}amilton-{J}acobi-{B}ellman
  equations.
\newblock {\em Nonlinear Analysis}, 70:1779--1796, 2009.

\bibitem{MV}
B.~Maslowski and P.~Veverka.
\newblock Infinite horizon maxmimum principle for the discounted control
  problem - incomplete version.
\newblock {\em arXiv}, 2011.

\bibitem{Mohammed}
S.-E.A. Mohammed.
\newblock Stochastic differential systems with memory: Theory, examples and
  applications.
\newblock In B.~{\O}ksendal L.~Decreusefond, Jon~Gjerde and A.S. Ustunel,
  editors, {\em Proceedings of The Sixth Workshop on Stchastic Analysis, Geilo,
  Norway, July 29-August 4, 1996, Stochastic Analysis and Related Topics VI.
  The Geilo Workshop}, pages 1--77. Progress in Probability, Birkh\"{a}user,
  1996.

\bibitem{MS}
S.-E.A. Mohammed and M.K.R. Scheutzow.
\newblock Lyapunov exponents of linear stochastic functional differential
  equations driven by semimartingales, part ii: Examples and case studies.
\newblock {\em Ann. of Prob.}, 6:1210 -- 1240, 1997.

\bibitem{OS}
B.~{\O}ksendal and A.~Sulem.
\newblock {\em A maximum principle for optimal control of stochastic systems
  with delay, with applications to finance: In J.M. Menaldi, E. Rofman and A.
  Sulem (editors): Optimal Control and Partial Differential Equations -
  Innovations and Applications}.
\newblock IOS Press, Amsterdam, 2000.

\bibitem{OS2}
B.~{\O}ksendal and A.~Sulem.
\newblock {\em {A}pplied {S}tochastic {C}ontrol of {J}ump {D}iffusions}.
\newblock Springer, 2nd edition, 2007.

\bibitem{OSZ}
B.~{\O}ksendal, A.~Sulem, and T.~Zhang.
\newblock Optimal control of stochastic delay equations and time-advanced
  backward stochastic differential equations.
\newblock {\em Adv. Appl. Prob.}, 43:572--596, 2011.

\bibitem{Pardoux}
E.~Pardoux.
\newblock Bsdes', weak convergence and homogenizations of semilinear pdes.
\newblock In F.H. Clark and R.J. Stern, editors, {\em Nonlinear Analysis,
  Differential Equations and Control}, pages 503--549. Kluwer Academic,
  Dordrecht, 1999.

\bibitem{PS}
S.~Peng and Y.~Shi.
\newblock Infinite horizon forward-backward stochastic differential equations.
\newblock {\em Stoch. Proc. and their Appl.}, 85:75--92, 2000.

\bibitem{BSDE}
M.~Royer.
\newblock Backward stochastic differential equations with jumps and related
  non-linear expectations.
\newblock {\em Stochastic Processes and Their Applications}, 116:1358--1376,
  2006.

\bibitem{Situ}
R.~Situ.
\newblock On solutions of backward stochastic differential equations with jumps
  and with non-{L}ipschitzian coefficients in {H}ilbert spaces and stochastic
  control.
\newblock {\em Statistics and Probability Letters}, 60:279--288, 2002.

\bibitem{Yin}
J.~Yin.
\newblock On solutions of a class of infinite horizon fbsdes.
\newblock {\em Statistics and Probability Letters}, 78:2412--2419, 2008.

\end{thebibliography}

\end{document}